\documentclass[12pt,leqno]{article}

\usepackage{mathrsfs}
\usepackage{amsmath}
\usepackage{amsthm}
\usepackage{amssymb}

\usepackage{amsfonts}
\usepackage{epsf}
\usepackage{graphicx}
\usepackage{hangcaption}
\usepackage{CJK}
\usepackage{hyperref}
\usepackage{amssymb}
\numberwithin{equation}{section}

\newtheorem{thm}{Theorem}

\newtheorem{lem}{Lemma}

\newtheorem*{rem}{Remark}

 \numberwithin{equation}{section}

\begin{document}
\title{Uniform Asymptotic Expansions for the Discrete Chebyshev Polynomials}
\author{J. H. Pan$^\text{a}$ and R. Wong$^\text{b}$}
\date{}
\maketitle  \noindent \emph{$^\text{a}$ Department of Mathematics,
City University of Hong
Kong, Tat Chee Avenue, Kowloon, Hong Kong\\
$^\text{b}$ Liu Bie Ju Centre for Mathematical Sciences, City
University of Hong Kong, Tat Chee Avenue, Kowloon, Hong Kong\\}
\begin{abstract}
The discrete Chebyshev polynomials $t_n(x,N)$ are orthogonal with
respect to a distribution function, which is a step function with
jumps one unit at the points $x=0,1,\cdots, N-1$, N being a fixed
positive integer. By using a double integral representation, we
derive two asymptotic expansions for $t_{n}(aN,N+1)$ in the double
scaling limit, namely, $N\rightarrow\infty$ and $n/N\rightarrow b$,
where $b\in(0,1)$ and $a\in(-\infty,\infty)$. One expansion involves
the confluent hypergeometric function and holds uniformly for
$a\in[0,\frac{1}{2}]$, and the other involves the Gamma function and
holds uniformly for $a\in(-\infty, 0)$. Both intervals of validity
of these two expansions can be extended slightly to include a
neighborhood of the origin. Asymptotic expansions for
$a\geq\frac{1}{2}$ can be obtained via a symmetry relation of
$t_{n}(aN,N+1)$ with respect to $a=\frac{1}{2}$. Asymptotic formulas
for small and large zeros of $t_{n}(x,N+1)$ are also given.
\end{abstract}

\newpage
\section{INTRODUCTION}
In Szeg$\ddot{o}$ \cite[p.33]{Szego}, the discrete Chebyshev
polynomials are defined by
\begin{equation}\label{DCP}
    t_n(x,N)=n!\triangle^n\binom{x}{n}\binom{x-N}{n},
    \qquad n=0,1,2,\cdot\cdot\cdot,N-1,
\end{equation}
where $\triangle$ denotes the difference operator
\begin{equation}
\begin{split}
  \triangle f(x)& = f(x+1)-f(x), \\
  \triangle^n f(x)& = \triangle\left\{\triangle^{n-1}f(x)\right\} \\
   &= f(x+n)-\binom{n}{1}f(x+n-1)+\cdot\cdot\cdot+(-1)^n f(x).
\end{split}
\end{equation}
These polynomials are orthogonal with respect to the distribution
$\mathrm{d}\alpha(x)$ of Stieltjes type, where $\alpha(x)$ is a step
function with jumps one unit at the points
$x=0,1,\cdot\cdot\cdot,N-1$, $N$ being a fixed positive integer.
More precisely, we have
\begin{equation}\label{}
    \int_{-\infty}^{\infty}t_n(x,N)t_m(x,N)\mathrm{d}\alpha(x)=\sum_{x=0,1,\cdot\cdot\cdot,N-1}t_n(x,N)t_m(x,N)=0
\end{equation}
if $n\neq m$, and
\begin{equation}
\begin{split}
\int_{-\infty}^{\infty}\left\{t_n(x,N)\right\}^2\mathrm{d}\alpha(x)&=\sum_{x=0,1,\cdot\cdot\cdot,N-1}\left\{t_n(x,N)\right\}^2\\
&=\frac{N(N^2-1^2)(N^2-2^2)\cdot\cdot\cdot(N^2-n^2)}{2n+1}
\end{split}
\end{equation}
for $n,m=0,1,2,\cdot\cdot\cdot,N-1$. The last two formulas hold for
all non-negative values of $n$ and $m$; in fact, they are trivial
for $n\geq N$ or $m\geq N$, since $t_n(x,N)=0$ for
$x=0,1,2,\cdot\cdot\cdot,N-1$, if
$n\geq N$.\\

The discrete Chebyshev polynomials can also be defined as a special
case of Hahn polynomials
\begin{equation}\label{definition of Q}
\begin{split}
Q_n(x;\alpha,\beta,N) &= {_3F_2}(-n,-x,n+\alpha+\beta+1;-N,\alpha+1;1) \\
   &= \sum_{k=0}^{n}\frac{(-n)_k(-x)_k(n+\alpha+\beta+1)_k}{(-N)_k(\alpha+1)_k
   k!},
\end{split}
\end{equation}
and we have
\begin{equation}\label{relation between DCP and Hahn}
\begin{split}
t_n(x,N)&=(-1)^n(N-n)_n Q_n(x;0,0,N-1)\\
&=(-1)^n(N-n)_n\sum_{k=0}^n\frac{(-n)_k(-x)_k(n+1)_k}{(-N+1)_k k!
   k!};
\end{split}
\end{equation}
see \cite[p.174 and 176]{Richard} and \cite{Chihara}. These
polynomials are used in least squares polynomial approximation
\cite[p.350]{Hildebrand} and sequential smoothing of numerical data
\cite{Morrison}. (In \cite{Hildebrand}, they are referred to as Gram
polynomials; see Chihara
\cite[p.162]{Chihara}.)\\

In this paper, we are concerned with the problem of finding
asymptotic expansions for the polynomials $t_n(x,N)$. Already, a
tremendous amount of research has been carried out on the
asymptotics of various discrete orthogonal polynomials. For
instance, we have \cite{Goh} and \cite{Rui} for Charlier
polynomials, \cite{xiaojin} and \cite{Wang} for Meixner polynomials,
\cite{Qiu} and \cite{Dai} for Krawtchouk polynomials, and \cite{Goh
and Wimp} and \cite{Lee} for Tricomi-Carlitz polynomials. There is
also an important research monograph by Baik et al. \cite{Baik} on a
general asymptotic method, based on the Riemann-Hilbert approach,
which in principle can be applied to all discrete orthogonal
polynomials. However, their result is so general that it does not
provide explicit asymptotic formulas for specific polynomials.
Strangely enough, for the discrete Chebyshev polynomials there does
not seem to be any asymptotic result in the literature; at least, we
have not been able to find any, and the
purpose of this paper is to fill in this gap.\\

For convenience, we shall consider the behavior of $t_n(x,N+1)$
instead of $t_n(x,N)$. Since the Hahn polynomials satisfy the
equation \cite[(1.15)]{symmetry}
\begin{equation}\label{}
    Q_n(N-x;\alpha,\beta, N)=\frac{Q_n(x;\alpha,\beta,N)}{Q_n(N;\alpha,\beta, N)}
\end{equation}
and $Q_n(N;0,0,N)=(-1)^n$ (note that there are slight differences in
the definition of $Q_n$ \cite[(1.1)]{symmetry}), the discrete
Chebyshev polynomials enjoy the symmetry relation
\begin{equation}\label{symmetricity}
    t_n(x,N+1)=(-1)^n t_n(N-x,N+1).
\end{equation}
In view of (\ref{symmetricity}), we may restrict $x$ to the interval
$-\infty<x\leq\frac{1}{2}N$. We are interested in the behavior of
$t_n(x,N+1)$ as $n, N\rightarrow\infty$ in such a way that the
ratios
\begin{equation}\label{a and b}
    a=x/N\qquad\text{and}\qquad b=n/N
\end{equation}
satisfy the inequalities
\begin{equation}\label{the range of a}
    -\infty<a\leq\frac{1}{2}\qquad\text{and}\qquad 0<b<1.
\end{equation}
As a special case of our general result, we will show that for fixed
value of $x\in[0,\infty)$, there exists a number $\eta<0$ such that
\begin{equation}\label{fix temp 1}
\begin{split}
t_n(x,N+1)=&\frac{(-1)^{n+1}\Gamma(n+N+2)N^x n^{-2x-2}\Gamma(x+1)}{\Gamma(N+1)\pi}\\
    &\times\left\{\sin\pi
    x\left[1+O\left(\frac{1}{N}\right)\right]+O\left(e^{\eta
    N}\right)\right\}.
\end{split}
\end{equation}
This result also holds when $x$ is negative, except that the term
$O(e^{\eta N})$ is absent.\\

The arrangement of the present paper is as follows. In \S\ref{sec
2}, we present two (2-dimensional) integral representations for
$t_n(x,N+1)$, one for the case $x<0$ and the other for $x\geq 0$. In
\S\ref{sec 3}, we transform the integration variables $(w,t)$ to two
new variables $(u,\tau)$, and deduce the double integrals to their
canonical forms, from which uniform asymptotic expansions can be
derived by a standard integration-by-parts procedure. In \S\ref{sec
4}, we prove that the transformation $(w,t)\rightarrow(u,\tau)$ is
one-to-one and analytic. In \S\ref{sec 5}, we investigate the
analyticity of the integrands of the double integrals in (\ref{IP
for x>0 deformed}) and (\ref{IR after mapping in x<0}). The desired
asymptotic expansions are then derived in \S\ref{sec 6}, and the
estimation of the error terms is done in \S\ref{sec 7}. In
\S\ref{sec 8}, we state the final results and give a list of special
cases; and we extend the regions of validity of the final results in
\S\ref{validity region}. In \S\ref{sec 9}, we present some
asymptotic formulas for the large and small zeros of the polynomial
under consideration.

\section{INTEGRAL REPRESENTATION}\label{sec 2}
From (\ref{relation between DCP and Hahn}), we have
\begin{equation}\label{integral temp 1}
\begin{split}
t_n(x,N+1)&=(-1)^n(N+1-n)_n Q_n(x;0,0,N)\\
&=(-1)^n(N+1-n)_n\sum_{k=0}^n\frac{(-n)_k(-x)_k(n+1)_k}{(-N)_k k!
   k!}.
\end{split}
\end{equation}
Using beta integral, we also have
\begin{equation}
\begin{split}
\frac{(n+1)_k}{(-N)_k} &= (-1)^k \frac{\Gamma(n+1+k)\Gamma(N+1-k)}{\Gamma(n+1)\Gamma(N+1)}\\
   &=(-1)^k\frac{\Gamma(n+N+2)}{\Gamma(n+1)\Gamma(N+1)}\int_0^1(1-t)^{n+k}t^{N-k}\rm{d}t\label{tempb}
\end{split}
\end{equation}
for $0\leq k\leq n\leq N$. Substituting (\ref{tempb}) in
(\ref{integral temp 1}) and interchanging the summation and
integration signs, we obtain
\begin{equation}\label{DCP and Meixner Polynomials}
\begin{split}
t_n(x,N+1)=&(-1)^n\frac{\Gamma(n+N+2)}{\Gamma(n+1)\Gamma(N+1-n)}\\
    &\times \int_0^1
(1-t)^n t^N{_2F_1}(-n,-x;1;1-t^{-1})\mathrm{d}t.
\end{split}
\end{equation}
In the above integral, we now use the new scales $a=x/N$ and $b=n/N$
introduced in (\ref{a and b}), and consider separate cases (i)
$0\leq a\leq\frac{1}{2}$ and (ii) $a<0$; cf. (\ref{the range of a}).\\

Recall the identity \cite[(15.8.1)]{handbook}
\begin{equation}\label{integral temp 2}
    _2F_1(-n,-x;1;1-t^{-1})=t^{-n}{_2F_1}(-n,1+x;1;1-t)
\end{equation}
and the contour integral representation \cite[p.287]{Richard}
\begin{equation}\label{integral temp 3}
_2F_1(a,b;c;z)=\frac{\Gamma(c)\Gamma(1+b-c)}{2\pi
i\Gamma(b)}\int_{\gamma_1} w^{b-1}
(w-1)^{c-b-1}(1-wz)^{-a}\mathrm{d}w,
\end{equation}
where $\text{Re } b>0$ and the curve $\gamma_1$ starts at $w=0$,
runs along the lower edge of the positive real line towards $w=1$,
encircles the point $w=1$ in the counterclockwise direction, and
returns to the origin along the upper edge of the positive real
line. For case (i) $0\leq a\leq\frac{1}{2}$, we apply (\ref{integral
temp 2}) and (\ref{integral temp 3}) to (\ref{DCP and Meixner
Polynomials}), and obtain the double integral representation
\begin{equation}\label{IR of DCP for x>0}
t_n(x,N+1)=\frac{(-1)^n}{2\pi
i}\frac{\Gamma(n+N+2)}{\Gamma(n+1)\Gamma(N-n+1)}\int_0^1\int_{\gamma_1}\frac{1}{w-1}
e^{N f(t,w)}\mathrm{d}w\mathrm{d}t,
\end{equation}
where
\begin{equation}\label{phase function f for x>0}
f(t,w)=b \ln(1-t)+(1-b)\ln t+a\ln w-a \ln(w-1)+b \ln[1-(1-t)w].
\end{equation}
The function $f(t,w)$ has a cut $(-\infty,1]$ in the $w$-plane.
Since the coefficients of $\ln t$, $\ln(1-t)$ and $\ln[1-(1-t)w]$ in
$N f(t,w)$ are all nonnegative integers, there is no need to have
cuts in the $t$-plane for $N f(t,w)$.\\

For case (ii), we use instead of (\ref{integral temp 3}) the contour
integral representation
\begin{equation}\label{integral temp 4}
_2F_1(a,b;c;z)=\frac{-\Gamma(c)\Gamma(1-b)}{2\pi
i\Gamma(c-b)}\int_{\gamma_2} (-w)^{b-1}
(1-w)^{c-b-1}(1-wz)^{-a}\mathrm{d}w,
\end{equation}
where $\text{Re } c> \text{Re } b$ and the curve $\gamma_2$ starts
at $w=1$, moves along the upper edge of the positive real line
towards $w=0$, encircles the origin in the counterclockwise
direction, and returns to $w=1$ along the lower edge of the positive
real line. (A proof of this formula can be given along the same
lines as that of (\ref{integral temp 3}).) Thus, we have
\begin{equation}\label{IR of DCP for x<0}
t_n(x,N+1)=\frac{(-1)^n}{2\pi
i}\frac{\Gamma(n+N+2)}{\Gamma(n+1)\Gamma(N-n+1)}\int_0^1\int_{\gamma_2}\frac{1}{w-1}e^{N
\widetilde{f}(t,w)}\mathrm{d}w\mathrm{d}t,
\end{equation}
where
\begin{equation}\label{phase function f for x<0}
\widetilde{f}(t,w)=b \ln(1-t)+(1-b)\ln t+a\ln(-w)-a \ln(1-w)+b
\ln[1-(1-t)w].
\end{equation}
Similar to $f(t,w)$ in (\ref{phase function f for x>0}), the
function $\widetilde{f}(t,w)$ has a cut $[0,+\infty)$ in the
$w$-plane and no cut in the $t$-plane.

\section{REDUCTION TO CANONICAL INTEGRALS}\label{sec 3}
To obtain large-$\lambda$ behavior of higher dimensional integral
\begin{equation*}
    I(\lambda)=\int_{\gamma}f(z)e^{\lambda S(z)}\mathrm{d}z,\qquad
    z=(z_1,z_2,\cdot\cdot\cdot,z_n)\in\mathbb{C}^n,
\end{equation*}
where $\mathrm{d}z=\mathrm{d}z_1\cdot\cdot\cdot\mathrm{d}z_n$ and
$\gamma$ is a smooth manifold of (real) dimension $n$, Fdeoryuk
\cite{Fedoryuk} showed that just like the classical method of
steepest descent, the saddle points of the phase function $S(z)$
(i.e., the zeros of the first-order derivatives of $S(z)$) play an
important role; see also Kaminski \cite{Kaminski}. So, let us first
investigate the location of the saddle points of the phase functions
$f(t,w)$ and $\widetilde{f}(t,w)$ in (\ref{phase function f for
x>0}) and
(\ref{phase function f for x<0}).\\

Upon solving the equations
\begin{equation*}\label{}
 \frac{\partial f}{\partial t}(t,w)=0\qquad\quad \text{and}\quad \qquad  \frac{\partial f}{\partial w}(t,w)=0
\end{equation*}
and the equations
\begin{equation*}\label{}
 \frac{\partial \widetilde{f}}{\partial t}(t,w)=0\qquad\quad \text{and}\quad \qquad  \frac{\partial \widetilde{f}}{\partial
 w}(t,w)=0,
\end{equation*}
one finds that the phase functions $f(t,w)$ and $\widetilde{f}(t,w)$
have the same two sets of saddle points

\begin{equation}\label{sets of saddle points 1}
  t_{+}=\frac{2-2a-b^2+ b\sqrt{b^2-4a+4a^2}}{2(1-a)(1+b)},\qquad  w_+=\frac{b+\sqrt{b^2-4a+4a^2}}{2b}
\end{equation}
and
\begin{equation}\label{sets of saddle points 2}
 t_{-}=\frac{2-2a-b^2-
 b\sqrt{b^2-4a+4a^2}}{2(1-a)(1+b)},\qquad  w_-=\frac{b-\sqrt{b^2-4a+4a^2}}{2b} .
\end{equation}
Note that $w_+$ and $w_-$, as well as $t_+$ and $t_-$, coalesce when
$a$ approaches $a_+$ and $a_-$, where
\begin{equation}\label{critical value of a}
    a_\pm=\frac{1\pm\sqrt{1-b^2}}{2}.
\end{equation}
The following facts are easily verified:
\begin{enumerate}
  \item For $a<0$, $w_\pm$ are both real and $w_-<0<1<w_+$. As $a\rightarrow
  0$, $w_-\rightarrow 0^-$ and $w_+\rightarrow 1^+$.
  \item For $0< a<a_-$, both $w_+$ and $w_-$ are real and
  $0<w_-<\frac{1}{2}<w_+<1$. As $a\rightarrow a_-$, both $w_+$ and $w_-$ approach $\frac{1}{2}$.
  \item For $a_-<a<a_+$, $w_\pm$ are complex conjugates, the real
  parts of which are both equal to $\frac{1}{2}$. At $a=\frac{1}{2}$, $w_\pm$
  attain their highest and lowest point on the line
  $\text{Re }{w}=\frac{1}{2}$.
  \item The locations of $w_\pm$ in the cases $a_+<a<1$ and $a>1$ are similar to those in the
  cases
  $0<a<a_-$ and $a<0$, respectively. This fact also reflects in the symmetry property of the discrete Chebyshev polynomial $t_n(aN, N+1)$ in $a$ with respect to $a=\frac{1}{2}$;
  cf. (\ref{symmetricity}) and (\ref{the range of a}).
\end{enumerate}
Figure \ref{w+-} summarizes the movements of $w_\pm$.\\

\begin{figure}[!htbp]
\centering
   \includegraphics[scale=0.65]{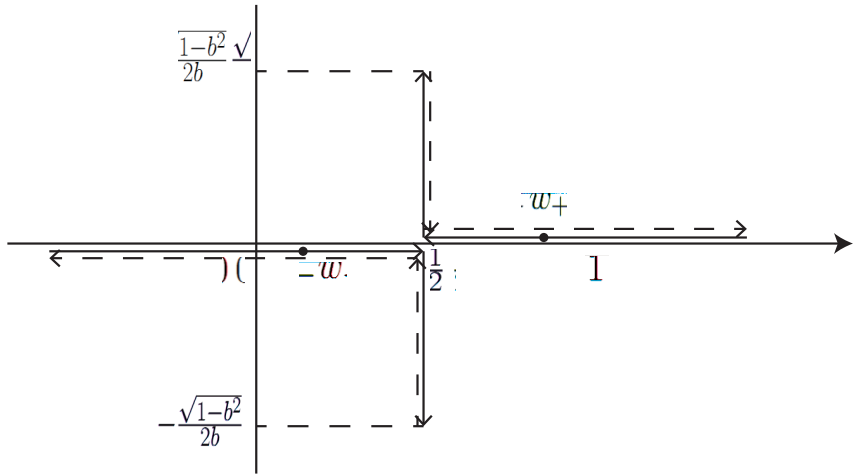}\\
 \caption{Movements of $w_\pm$.}\label{w+-}
\end{figure}

In view of (\ref{the range of a}), we only need consider the case
$-\infty<a\leq\frac{1}{2}$. We again divide our discussion into two
subcases: (i) $0\leq a\leq\frac{1}{2}$ and (ii) $-\infty<a<0$.\\

{\textbf{Case (i)} $0\leq a\leq\frac{1}{2}$.\\

Returning to the integral representation (\ref{IR of DCP for x>0}),
we now show that this double integral can be reduced to a canonical
form, and that the original integration surface
$[0,1]\times\gamma_1$ in the $t\times w$ plane can be deformed into
a new integration surface. An outline of the procedure is as
follows. First, for each fix $w\in\gamma_1$, we find a steepest
descent path of the variable $t$, which passes through a relevant
saddle point $t_0(w)$ depending on $w$. We denote this path by
$\rho_t$. Note that not only the saddle point $t_0(w)$, but all
points $t$ on the path $\rho_t$ depend on $w$. Thus, the phase
function $f(t_0(w),w)$ is now a function of $w$ alone. Second,
deform the integration path $\gamma_1$ in (\ref{IR of DCP for x>0})
into a steepest descent path of $w$, passing through $w_\pm$ given
in (\ref{sets of saddle points 1}) and (\ref{sets of saddle points
2}). We denote this path by $\gamma_w$. It will be shown that these
saddle points are indeed saddle points of $f(t_0(w),w)$.\\

To find the relevant saddle point $t_0(w)$, we solve the equation
\begin{equation*}
\frac{\partial}{\partial t}f(t,w)=0,
\end{equation*}
and obtain
\begin{equation}\label{saddle points for t}
    t_0^{\pm}(w)=\frac{2w-1\pm\sqrt{1+4b^2(w-1)w}}{2(1+b)w}.
\end{equation}
Here, we have introduced two more cuts in the $w$-plane, both on the
vertical line $\text{Re }w=\frac{1}{2}$, one from
$(b+i\sqrt{1-b^2})/2b$ all the way to $w=\frac{1}{2}+i\infty$, and
the other from $(b-i\sqrt{1-b^2})/2b$ all the way to
$w=\frac{1}{2}-i\infty$. As remarked earlier, the two saddle points
$(t_+,w_+)$ and $(t_-,w_-)$ in (\ref{sets of saddle points 1}) and
(\ref{sets of saddle points 2}) play an important role in
determining the asymptotic behavior of $t_n(x,N+1)$. So, we must be
careful in choosing the right saddle point from the two given in
(\ref{saddle points for t}). For $a\leq\frac{1}{2}$, it can be shown
that $t_0^+(w_\pm)=t_\pm$. Thus, $t_0^+(w)$ is our relevant saddle
point. Throughout the following discussion, we shall simply write $t_0(w)$ for $t_0^+(w)$, when no misunderstanding would arise.\\

As we have explained previously, for fixed $w$ in the whole
$w$-plane and $b\in(0,1)$, the function $e^{Nf(t,w)}$ inside the
integral in (\ref{IR of DCP for x>0}) has no singularity in the
$t$-plane. However, to make it possible for us to deform the
original integration path $[0,1]$ into a steepest descent path
$\rho_t$, we have to define the function $f(t,w)$ in such a way that
it avoids the cut $(-\infty,1]$ in the $w$-plane, and three
additional cuts in the $t$-plane, $(-\infty,0]$, $[1,\infty)$ and
also the radial line starting from $1-1/w$ and satisfying
$1-(1-t)w<0$. With this in mind, we now define the steepest descent
path $\rho_t$, passing through $t_0(w)$, by
\begin{equation}\label{SDP equation for t}
\begin{split}
   \text{Im }{f(t,w)}&=\text{Im }{f(t_0(w),w)},\\
   \text{Re }{f(t,w)}&\leq\text{Re }{f(t_0(w),w)};
   \end{split}
\end{equation}
see Figure \ref{SDP of t fig 1} and Figure \ref{SDP of t fig 2}. It
is now standard to define a transformation $t\rightarrow\tau$ by
\begin{equation}\label{mapping t to tau}
    \tau^2=f(t_0(w),w)-f(t,w).
\end{equation}
\begin{figure}
\centering
\begin{minipage}[t]{0.4\textwidth}
\centering
  \includegraphics[width=100pt,bb=0 0 200 200]{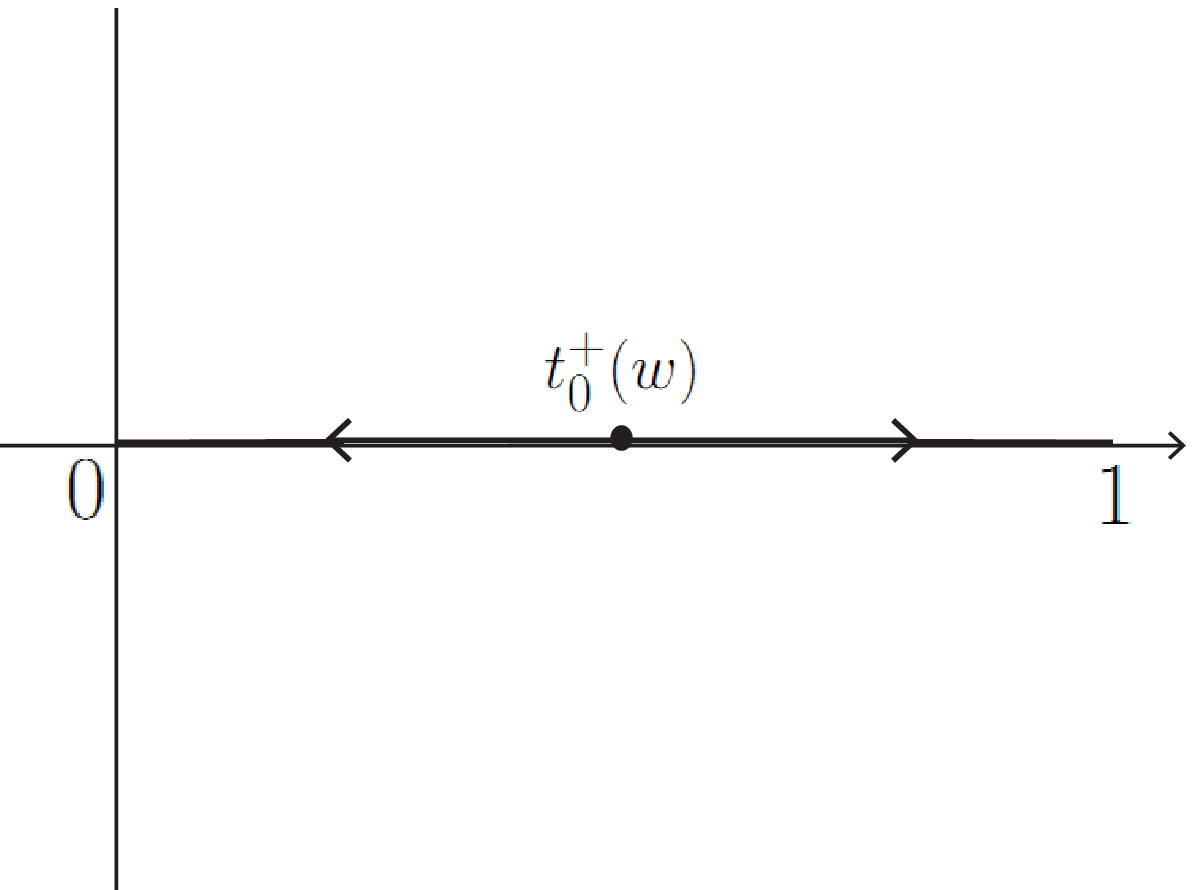}\\
  \caption{Steepest descent path in the $t$-plane when $w$ is real.}\label{SDP of t fig 1}
\end{minipage}\hfill
\begin{minipage}[t]{0.45\textwidth}
\centering
  \includegraphics[scale=0.5]{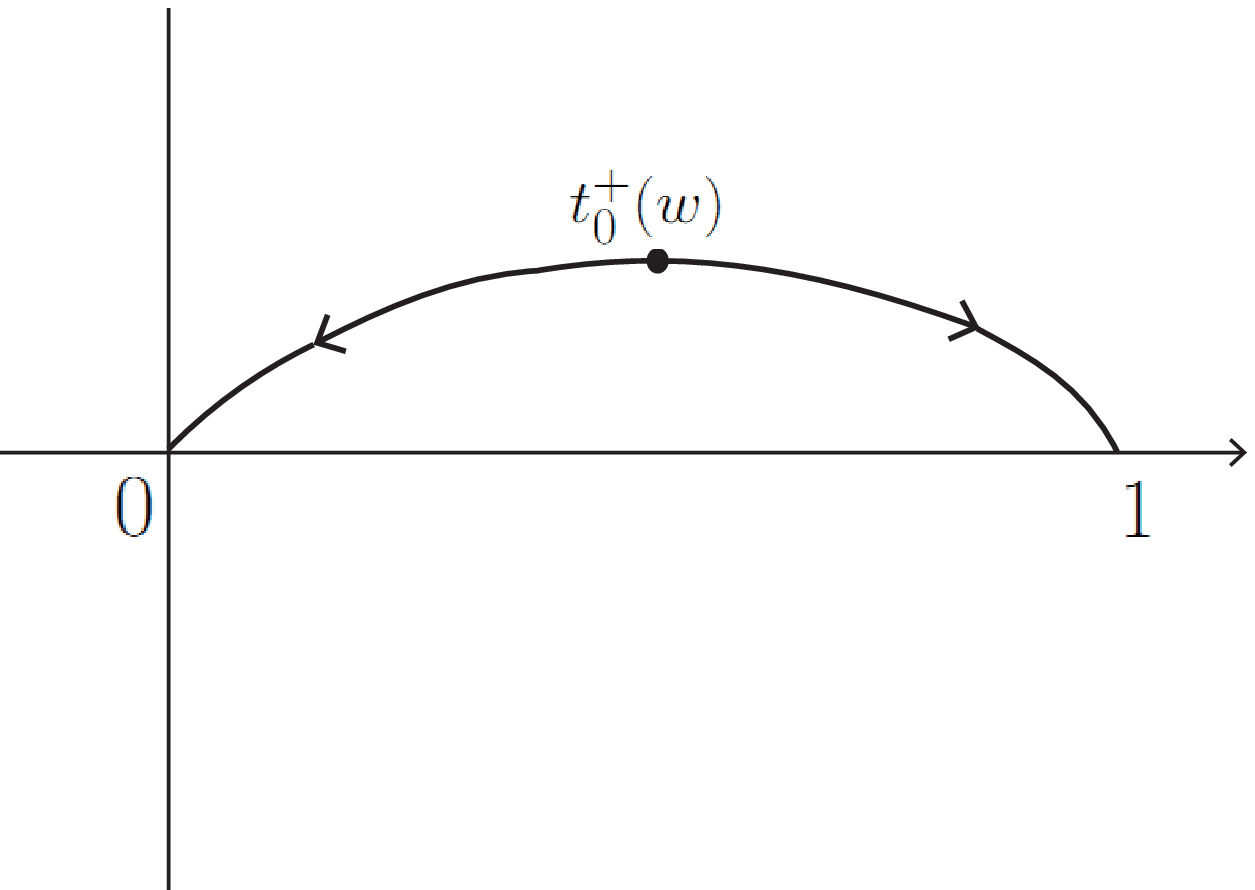}\\
  \caption{Steepest descent path in the $t$-plane when $w$ is complex; e.g., $w=0.5+0.2i$.}\label{SDP of t fig 2}
\end{minipage}\hfill
\end{figure}
\vspace{-0.6cm}\newline \noindent Note that we have $\tau=-\infty$
when $t=0$, and $\tau=+\infty$ when $t=1$. Furthermore, this mapping
is one-to-one with $\tau$ running from $-\infty$ to $\infty$, and
takes $t=t_0(w)$ to $\tau=0$. From $(\ref{phase function f for
x>0})$ and $(\ref{mapping t to tau})$, we have
\begin{equation}\label{dt over dtau}
        \begin{split}
            \frac{dt}{d\tau}& =\frac{2\tau}{-\partial f(t,w)/\partial
            t}=\frac{2\tau t(1-t)[1-(1-t)w]}{(1+b)w(t-t^+_0(w))(t-t^-_0(w))},\quad\quad \tau\neq
            0,\\
            \frac{dt}{d\tau}&=\left\{\frac{2t^+_0(w)(1-t^+_0(w))[1-(1-t^+_0(w))w]}{\sqrt{1+4b^2(w-1)w}}\right\}^{1/2}, \hspace{1.6cm} \tau=0.
        \end{split}
\end{equation}
Coupling (\ref{IR of DCP for x>0}) and (\ref{mapping t to tau})
gives
\begin{equation}\label{IP after the t mapping}
\begin{split}
t_n(x,N+1)=& \frac{(-1)^n}{2\pi
i}\frac{\Gamma(n+N+2)}{\Gamma(n+1)\Gamma(N-n+1)}\\
&\times \int_{-\infty}^{+\infty}\int_{\gamma_1}\frac{1}{w-1}e^{N
f({t_0(w)},w)}e^{-N\tau^2}\frac{\mathrm{d}t}{\mathrm{d}
\tau}\mathrm{d}w\mathrm{d}\tau,
\end{split}
\end{equation}
where $\gamma_1$ is the integration path of $w$ in (\ref{IR of DCP
for x>0}).
Note that the first variable of $f$ in (\ref{IP after the t
mapping}) is now $t_0(w)$, instead of $t$. So the phase function
depends only on $w$. Setting $\partial f(t_0(w),w)/\partial w=0$, we
obtain saddle points
%
\begin{equation}\label{saddle points of w}
    w_{\pm}=\frac{b\pm\sqrt{b^2-4a+4a^2}}{2b};
\end{equation}
cf. (\ref{sets of saddle points 1}) and (\ref{sets of saddle points
2}). We now deform $\gamma_1$ into steepest descent paths passing
through $w_+$ and $w_-$. Recall that the movements of $w_\pm$ are
summerized in Figure \ref{w+-}, as $a$ runs along the real line. In
the case $a_-<a\leq\frac{1}{2}$, the saddle points $w_\pm$ are
complex conjugate numbers with real part equal to $\frac{1}{2}$. The
steepest descent path passing through $w_+$ is given by
\begin{equation}\label{SDP equation for w+}
    \begin{split}
    \text{Im }{f(t_0(w),w)}&= \text{Im }{f(t_0(w_+),w_+)},\\
        \text{Re }{f(t_0(w),w)}&\leq \text{Re }{f(t_0(w_+),w_+)}.
        \end{split}
\end{equation}
The steepest descent path through $w_-$ can be obtained by using the
symmetry property with respect to the real axis; see Figure
\ref{w-sdp1}.\\
\begin{figure}
\centering
  \includegraphics[scale=0.5]{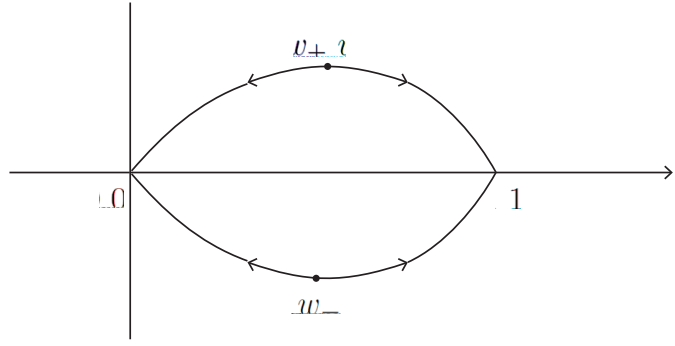}\\
  \caption{Steepest descent path of $w$ when $a_-<a\leq\frac{1}{2}$.}\label{w-sdp1}
\end{figure}
\begin{figure}
\centering
  \includegraphics[scale=0.5]{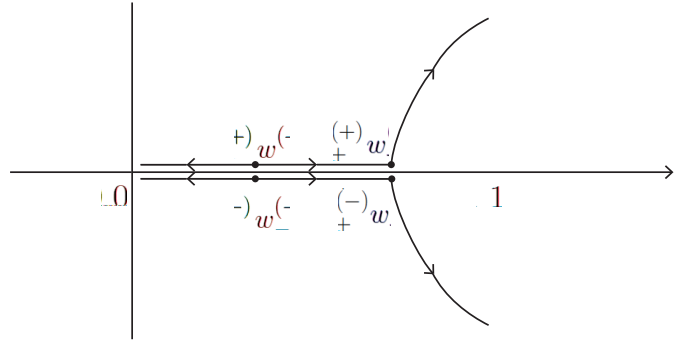}\\
  \caption{Steepest descent path of $w$ when $0\leq a<a_-$.}\label{w-sdp2}
\end{figure}

When $0\leq a< a_-$, both $w_\pm$ are real and lie in the interval
$[0,1]$; thus, $w_\pm$ appear on both upper and lower edges of the
cut along the real line. We denote these points by $w_\pm^{(+)}$ and
$w_\pm^{(-)}$; see Figure \ref{w-sdp2}. Since $\text{Re
}f(t_0(w_-^{(\pm)}),w_-^{(\pm)})>\text{Re
}f(t_0(w_+^{(\pm)}),w_+^{(\pm)})$, $w_-^{(\pm)}$ is a local maximum
and the steepest descent path through $w_-^{(\pm)}$ is along the
real axis. From Figure \ref{w+-}, it is clear that this path is on
the real interval $(0, w_+^{(\pm)})$. The steepest descent path
throught $w_+^{(\pm)}$ is the curve perpendicular to the real axis;
cf. Figure \ref{w-sdp2}.\\

Motivated by the work of Jin and Wong \cite{xiaojin}, the movements
of the saddle points $w_\pm$ and the steepest descent paths passing
through them suggest that we should compare the integral of $w$ in
(\ref{IP after the t mapping}) with the confluent hypergeometric
function \cite[p.326, (13.4.9)]{handbook}
\begin{equation}\label{m function a>0}
    M(d,c,z)=\frac{\Gamma(c)\Gamma(1+d-c)}{2\pi
    i\Gamma(d)}\int_{\gamma_1}u^{d-1}(u-1)^{c-d-1}e^{zu}\mathrm{d}u,
\end{equation}
where $\text{Re }{d}>0$ and the integration curve $\gamma_1$ is the
same one as that given in (\ref{integral temp 3}) or (\ref{IP after
the t mapping}). Since $M(d,c,z)$ is a meromorphic function of $c$
with poles at 0, $-1$, $-2$,$\cdot\cdot\cdot$, the function
$\textbf{M}(d,c,z):=M(d,c,z)/\Gamma(c)$ is entire in both $d$ and
$c$; see \cite[p.255]{Olver}. With $d=aN+1$ ($0\leq a\leq
\frac{1}{2}$), $c=1$, $z=\eta N$, (\ref{m function a>0}) gives
\begin{equation}\label{confluent hypergeometric function for a>0}
    \textbf{M}(aN+1,1,\eta N)=\frac{1}{2\pi
    i}\int_{\gamma_1}e^{N\psi(u)}\frac{\mathrm{d}u}{u-1},
\end{equation}
where
\begin{equation}\label{psi function}
    \psi(u)=a\ln u-a\ln(u-1)+\eta u
\end{equation}
and the saddle points of $\psi$ are given by
\begin{equation}\label{saddle points of u}
    u_\pm=\frac{\eta\pm\sqrt{\eta^2+4a\eta}}{2\eta}.
\end{equation}
The shape of the steepest descent paths of $\psi$ depends on the
sign of $a$ and the value of $\eta$. Here, we are concerned with
only the two cases: (i) $a\geq 0$ and $-4a<\eta<0$, (ii) $a\geq 0$
and
$\eta<-4a$; see Figures \ref{u-sdp1} and \ref{u-sdp2}.\\
\begin{figure}
\centering
  \includegraphics[scale=0.5]{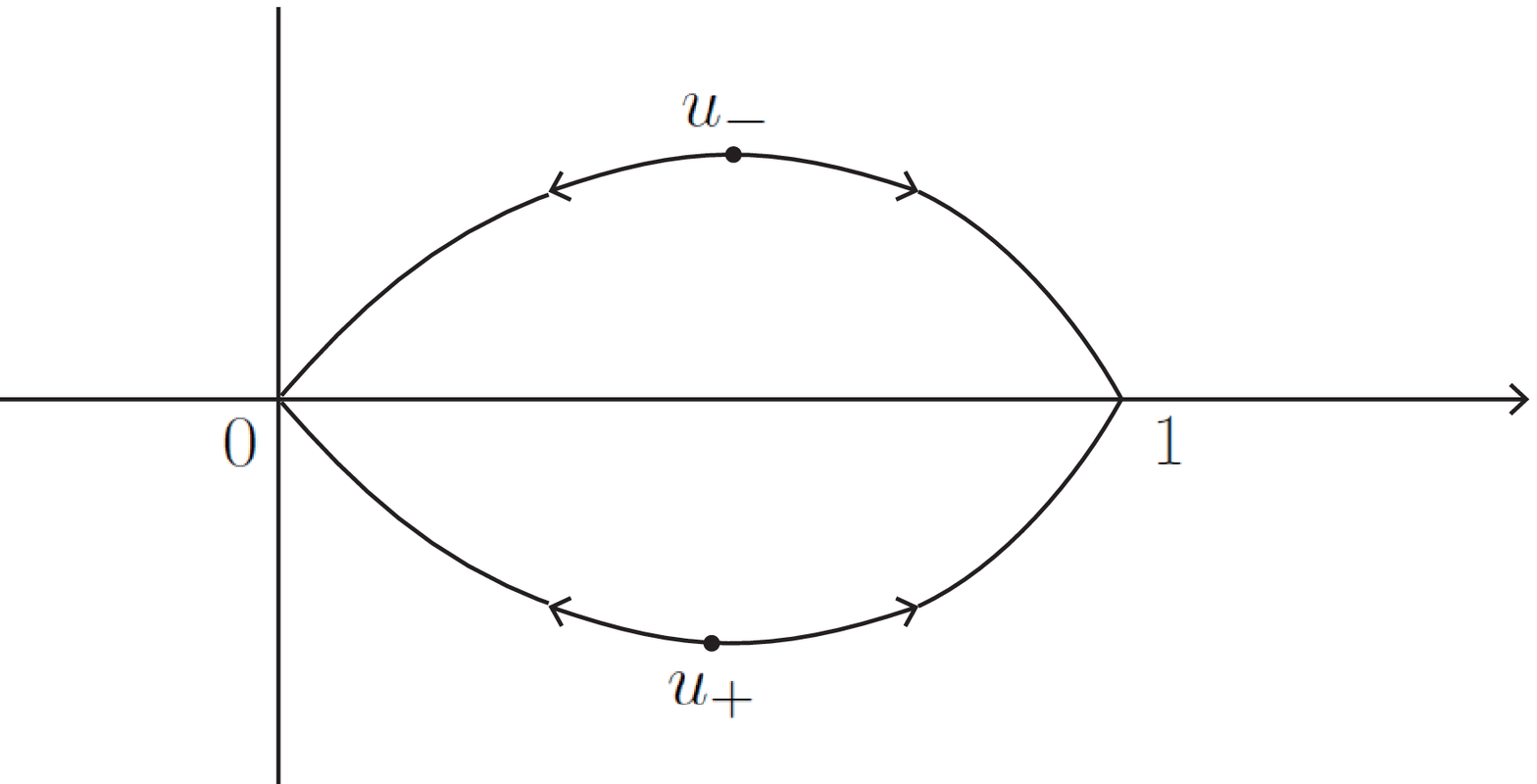}\\
  \caption{Steepest descent path of $u$ when $a>0$ and $-4a<\eta<0$.}\label{u-sdp1}
\end{figure}
\begin{figure}
\centering
  \includegraphics[scale=0.5]{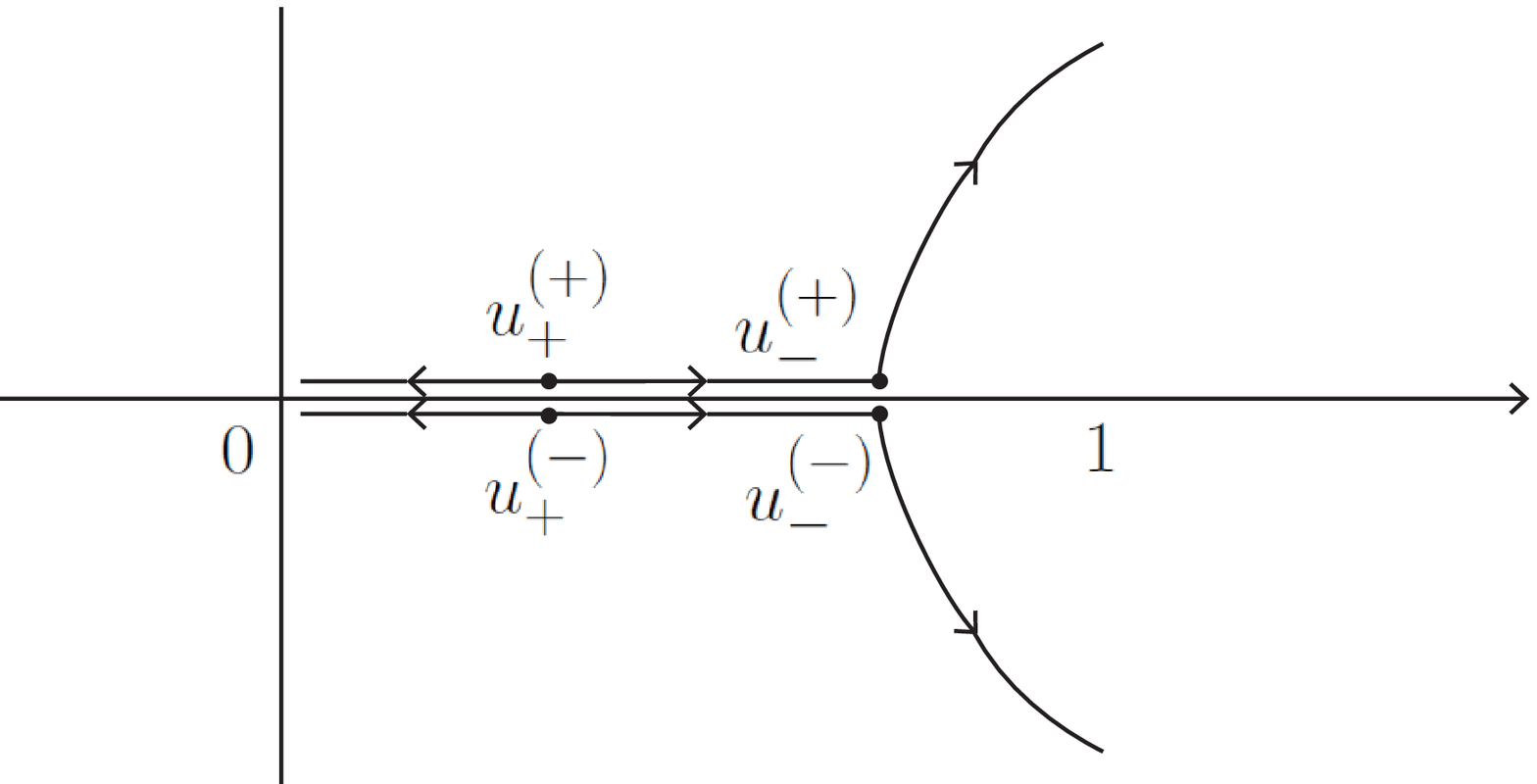}\\
  \caption{Steepest descent path of $u$ when $a>0$ and $\eta<-4a$.}\label{u-sdp2}
\end{figure}

We also observe that when $a\geq 0$, the movements of $u_\pm$ are
like those of $w_\mp$, with the corresponding subcases
\begin{eqnarray*}
  a_-<a<\frac{1}{2} &\leftrightarrow& -4a<\eta<0, \nonumber\\
  0\leq a\leq a_- &\leftrightarrow&  \eta\leq -4a. \label{the correspondence of
  subcases}
\end{eqnarray*}\label{}
Thus, we define the mapping $w\rightarrow u$ by
\begin{equation}
\begin{split}
    f(t_0(w),w)&=\psi(u)+\gamma\\
    &=a \ln u-a\ln(u-1)+\eta u+\gamma\label{mapping w to u}
\end{split}
\end{equation}
with
\begin{equation}\label{the relation between w+- and u+-}
    u(w_+)=u_-,\quad\quad\quad\quad u(w_-)=u_+,
\end{equation}
where $\eta$ and $\gamma$ are real numbers. Note that since $w_\pm$
are on both edges of the real line (similarly, $u_\pm$ are also on
both edges of the real line) when $0\leq a\leq a_-$, (\ref{the
relation between w+- and u+-}) should be understood in the sense:
\begin{subequations}
\begin{align}
    u(w_+^{(+)})&=u_-^{(+)},\quad\quad\quad\quad
    u(w_-^{(+)})=u_+^{(+)},\label{correspondence on upper edge}\\
    u(w_+^{(-)})&=u_-^{(-)},\quad\quad\quad\quad
    u(w_-^{(-)})=u_+^{(-)}.\label{correspondence on lower edge}
\end{align}
\end{subequations}
It can be easily verified that (\ref{correspondence on upper edge})
and (\ref{correspondence on lower edge}) are equivalent in the case
$0\leq a\leq a_-$. The existence and uniqueness of the numbers
$\eta$ and $\gamma$ as well as the one-to-one and analytic property
of the mapping $w\rightarrow u$ defined in (\ref{mapping w
to u}) will be established in the following section.\\

Note that we have chosen the coefficients of $\ln u$ and $\ln(u-1)$
in $\psi(u)$ to be $a$. The reason is that with this choice, the
local expansions of $f(t_0(w),w)$ at $w=0$ and $w=1$ and that of
$\psi(u)+\gamma$ at the corresponding points $u=0$ and $u=1$ are the
same.\\

From (\ref{mapping w to u}), we have
\begin{equation}
\begin{split}
\frac{\mathrm{d} w}{\mathrm{d} u} &= \frac{\mathrm{d}
\psi}{\mathrm{d} u}\left/\frac{\mathrm{d}
  f(t_0(w),w)}{\mathrm{d} w}\right.
 \\
   &= \frac{\eta(u-u_+)(u-u_-)w(1-w)\left[(2a-1)-\sqrt{1+4b^2w^2-4b^2w}\right]}{-2b^2u(u-1)(w-w_+)(w-w_-)}\label{dw over du, normal},
\end{split}
\end{equation}
for $u\neq u_\pm$, and by l'H$\hat{o}$pital's rule,
\begin{equation}\label{dw over du, saddle point}
    \frac{\mathrm{d} w}{\mathrm{d}
    u}=\left[\frac{\sqrt{\eta^2+4a\eta}(1-2a)(1-a)(-\eta)}{b^3\sqrt{b^2-4a+4a^2}}\right]^{1/2}
\end{equation}
for $u=u_\pm$. (Note that $\sqrt{1+4b^2w_\mp^2-4b^2w_\mp}=1-2a$). \\

By coupling (\ref{IP after the t mapping}) and (\ref{mapping w to
u}), the double integral representation of $t_n(x,N+1)$ given in
(\ref{IR of DCP for x>0}) is now reduced to the canonical form
\begin{equation}\label{IP for x>0 deformed}
\begin{split}
 t_n(x,N+1)=&\frac{(-1)^n}{2\pi
i}\frac{\Gamma(n+N+2)}{\Gamma(n+1)\Gamma(N-n+1)}e^{N\gamma}\\
&\qquad\times
\int_{-\infty}^{+\infty}\int_{\gamma_u}\frac{h(u,\tau)}{u-1}e^{N
\psi(u)-N\tau^2}\mathrm{d}u\mathrm{d}\tau,
\end{split}
\end{equation}
where
\begin{equation}\label{h function}
h(u,\tau)=\frac{u-1}{w-1}\frac{\mathrm{d} w}{\mathrm{d}
u}\frac{\mathrm{d} t}{\mathrm{d} \tau}
\end{equation}
and $\gamma_u$ is the steepest descent path depicted in Figures
\ref{u-sdp1} and \ref{u-sdp2} when $a_-<a\leq\frac{1}{2}$ and
$0\leq a<a_-$, respectively.\\

Before concluding our discussion of Case (i): $0\leq
a\leq\frac{1}{2}$, let us give a summary of the reduction process
from (\ref{IR of DCP for x>0}) to (\ref{IP for x>0 deformed}). We
first deformed the original surface
$[0,1]\times\gamma_1\in\mathbb{C}^2$ in (\ref{IR of DCP for x>0})
into $\rho_t\times \gamma_w\in\mathbb{C}^2$, where $\rho_t$ is the
steepest descent path in the $t$-plane of the phase function
$f(t,w)$ in (\ref{IR of DCP for x>0}), which passes through the
saddle point $t_0(w)=t_0^+(w)$ in (\ref{saddle points for t}) for
each fixed $w$, and where $\gamma_w$ is the steepest descent path of
$f(t_0(w),w)$ in the $w$-plane, which passes through the saddle
points $w_\pm$ given in (\ref{saddle points of w}). Next, for every
fixed $w$, we introduced a one-to-one analytic map $t\rightarrow
\tau$, defined by (\ref{mapping t to tau}), which maps $\rho_t$ to
$(-\infty,+\infty)$, and a map $w\rightarrow u$, defined by
(\ref{mapping w to u}), which maps $\gamma_w\rightarrow\gamma_u$,
where $\gamma_u$ is the steepest descent path in the $u$-plane for
the canonical integral (\ref{IP for x>0 deformed}). In the sense of
Kaminski \cite{Kaminski}, we may refer to $S_{\tau\times
u}:=(-\infty,\infty)\times \gamma_u$ as a steepest descent surface
in $\mathbb{C}^2$. Thus, the 2-dimensional transformation
$(t,w)\rightarrow(\tau,u)$ defined by
\begin{equation}\label{two dimensional mapping}
\begin{split}
    b \ln(1-t)+(1-b)\ln t&+a\ln w-a \ln(w-1)+b \ln[1-(1-t)w]\\
    &=-\tau^2+a\ln u-a\ln(u-1)+\eta u+\gamma
\end{split}
\end{equation}
is, in fact, given in two stages; see (\ref{phase function f for
x>0}), (\ref{mapping t to tau}) and (\ref{mapping w to u}). This
mapping $(t,w)\rightarrow(\tau,u)$ takes $\rho_t\times \gamma_w$ in
the $\mathbb{C}^2$-plane $(t,w)$ into a steepest descent surface
$S_{\tau\times u}$ in the $\mathbb{C}^2$-plane $(\tau,u)$. Hence, in
conclusion, we have made the following change of surfaces of
integration
\begin{equation*}
   [0,1]\times\gamma_1\in\mathcal{C}^2_{t\times w}\quad \longrightarrow \quad \rho_t\times \gamma_w \in\mathcal{C}^2_{t\times w} \quad \longrightarrow \quad S_{\tau\times u}\in\mathcal{C}^2_{\tau\times
   u}.
\end{equation*}

\textbf{Case (ii)} $-\infty<a<0.$\\

In this case, instead of (\ref{IR of DCP for x>0}), we use the
double integral representation (\ref{IR of DCP for x<0}). The
argument here is similar to that in Case (i). But, we note from
$\S\ref{sec 2}$ that when $a$ is negative, both $w_-$ and $w_+$ are
real and $w_-<0<1<w_+$. Taking account of the shape of the contour
$\gamma_2$ in (\ref{IR of DCP for x<0}), we need be concerned with
only the saddle point $w_-$, and not $w_+$.\\

Since $a$ is negative and hence less than $\frac{1}{2}$, the
relevant saddle point is again $t_0(w)=t_0^+(w)$; see the argument
following (\ref{saddle points for t}), except with $f$ replaced by
$\widetilde{f}$. Thus, (\ref{IR of DCP for x<0}) can be written as

\begin{equation}\label{IP after the t mapping a<0}
\begin{split}
t_n(x,N+1) =& \frac{(-1)^n}{2\pi
i}\frac{\Gamma(n+N+2)}{\Gamma(n+1)\Gamma(N-n+1)}\\
& \times
\int_{-\infty}^{+\infty}\int_{\gamma_2}\frac{1}{w-1}e^{N\cdot
\widetilde{f}({t_0(w)},w)}e^{-N\tau^2}\frac{\mathrm{d} t}{\mathrm{d}
\tau}\mathrm{d}w\mathrm{d}\tau, \end{split}
\end{equation}
where $\gamma_2$ is the integration path of $w$ in (\ref{IR of DCP
for x<0}); cf. (\ref{IP after the t mapping}). However, the mapping
$w\rightarrow u$ differs from that in Case (i). It can be shown that
the second derivative of the phase function
$\widetilde{f}(t_0(w),w)$ is positive at $w=w_-$, and the steepest
descent path passing through $w_-$ is perpendicular to the real
axis; see Figure \ref{w-sdp3}.
\begin{figure}[h]
\centering
  \includegraphics[scale=0.5]{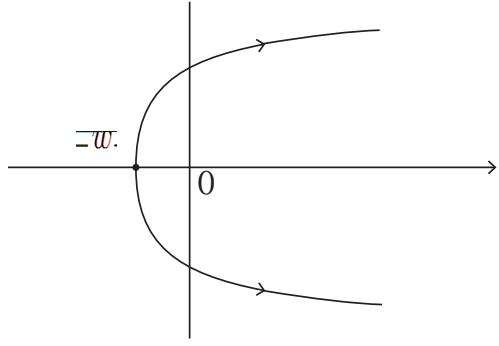}\\
  \caption{The steepest descent path of $w$, when $a<0$.}\label{w-sdp3}
\end{figure}

Recall the Hankel integral for the Gamma function
\begin{equation}\label{}
    \frac{1}{\Gamma(z)}=\frac{1}{2\pi
    i}\int^{(0+)}_{-\infty}e^uu^{-z}\mathrm{d}u,
\end{equation}
where the contour is a loop starting at $-\infty$, encircling the
origin in the counterclockwise direction, and returning to
$-\infty$. With $z=-a N+1$ and $u$ replaced by $-u$, we obtain
\begin{equation}\label{gamma function in use}
    \frac{1}{\Gamma(-a N+1)}=\frac{N^{aN}}{2\pi
    i}\int^{(0+)}_{\infty}\frac{1}{u}e^{N\widetilde{\psi}(u)}\mathrm{d}u,
\end{equation}
where
\begin{equation}\label{psi function when a<0}
    \widetilde{\psi}(u)=a\ln(-u)-u,
\end{equation}
and the saddle point of $\widetilde{\psi}$ is at $u=a$.\\

Observe the similarities of the saddle points and the integration
paths (see Figure \ref{w-sdp3}) in the two integrals (\ref{IP after
the t mapping a<0}) and (\ref{gamma function in use}). Make the
transformation $w\rightarrow u$ defined by
\begin{equation}\label{mapping w to u when x<0}
    \begin{split}
    \widetilde{f}(t_0(w),w)&=\widetilde{\psi}(u)+\gamma\\
    &=a\ln (-u)-u+\gamma
    \end{split}
\end{equation}
with
\begin{equation}\label{relation betwen w and u when a<0}
    u(w_-)=a,
\end{equation}
where $\gamma$ is a constant to be determined.\\

Coupling (\ref{IP after the t mapping a<0}) and (\ref{mapping w to u
when x<0}) gives
\begin{equation}\label{IR after mapping in x<0}
\begin{split}
    t_n(x,N+1)=&\frac{(-1)^n}{2\pi
i}\frac{\Gamma(n+N+2)}{\Gamma(n+1)\Gamma(N-n+1)}\\
& \times
e^{N\gamma}\int_{-\infty}^{+\infty}\int_{+\infty}^{(0+)}\frac{h(u,\tau)}{u}
e^{N \widetilde{\psi}(u)-N\tau^2}\mathrm{d}u\mathrm{d}\tau,
\end{split}
\end{equation}
where
\begin{equation}\label{h function when a<0}
    h(u,\tau)=\frac{u}{w-1}\frac{\mathrm{d} w}{\mathrm{d}
u}\frac{\mathrm{d} t}{\mathrm{d} \tau}.
\end{equation}
From (\ref{mapping w to u when x<0}), we have
\begin{equation}
\begin{split}
  \frac{\mathrm{d} w}{\mathrm{d} u}&=\frac{a-u}{u}\left/\frac{\mathrm{d} \widetilde{f}(t_0(w),w)}{\mathrm{d}
  w}\right.\\
  &=\frac{(a-u)w(1-w)\left[(1-2a)+\sqrt{1+4b^2w^2-4b^2w}\right]}{2ub^2(w-w_+)(w-w_-)}
  \label{dw over du when a<0}
  \end{split}
\end{equation}
for $u\neq a$, and by l'H$\hat{o}$pital's rule
\begin{equation}\label{dw over du when a<0 saddle point}
  \frac{\mathrm{d} w}{\mathrm{d} u} =
  \left\{\frac{(1-a)(1-2a)}{b^3\sqrt{b^2-4a+4a^2}}\right\}^{1/2}
\end{equation}
for $u=a$.

\section{THE MAPPING $w\rightarrow u$ IN (\ref{mapping w to u})}\label{sec 4}

Before establishing the existence and uniqueness of the constants
$\eta$ and $\gamma$ in equation (\ref{mapping w to u}) which defines
the mapping $w\rightarrow u$, we first need a preliminary result.

\begin{lem}\label{lem 1}
Let $0\leq a\leq\frac{1}{2}$, and put $g(\eta):=\psi(u_-)-\psi(u_+)$
when $a_-<a\leq\frac{1}{2}$, where
$u_\pm=(\eta\pm\sqrt{\eta^2+4a\eta})/2\eta$ are given in
(\ref{saddle points of u}); put
$g_\pm(\eta):=\psi(u_-^{(\pm)})-\psi(u_+^{(\pm)})$ when $0\leq a\leq
a_-$, where $u_\pm^{(+)}$ denote the $u_\pm$ on the upper edge of
the real line and $u_\pm^{(-)}$ denote the $u_\pm$ on the lower edge
of the real line; see Figure \ref{u-sdp2}. We have
\begin{itemize}
  \item[(i)] when $\eta\leq -4a$, $g_+(\eta)=g_-(\eta)=k(\eta)$, where
$k(\eta)$ is a real and monotonically increasing function with range
$(-\infty,0]$;
\item[(ii)] when $-4a<\eta<0$, $g(\eta)$ is purely imaginary and $-i
g(\eta)$ is a monotonically increasing and continuous function with
range $(-2a\pi,0)$.
\end{itemize}
\end{lem}

\begin{proof}
In case (i), we have $\eta^2+4a\eta>0$. By definition,
\begin{equation*}
    g_\pm(\eta)=a\ln u_-^{(\pm)}-a\ln(u_-^{(\pm)}-1)+\eta u_-^{(\pm)}-a\ln u_+^{(\pm)}+a\ln(u_+^{(\pm)}-1)-\eta
    u_+^{(\pm)}.
\end{equation*}
Recall the movements of $u_\pm$ shown in Figure \ref{u-sdp2}. Since
$u_\pm^{(-)}$ lie on the lower edge of the cut $[0,1]$, it follows
that
\begin{equation*}
    u_\pm^{(-)}-1 = |u_\pm^{(-)}-1|e^{-\pi i}=\frac{\eta-\sqrt{\eta^2+4a\eta}}{2\eta}e^{-\pi
    i}.
\end{equation*}
Similarly, since $u_\pm^{(+)}$ lie on the upper edge of the cut, we
have
\begin{equation*}
    u_\pm^{(+)}-1 =|u_\pm^{(+)}-1|e^{\pi i}=\frac{\eta+\sqrt{\eta^2+4a\eta}}{2\eta}e^{\pi
    i}.
\end{equation*}
Put
\begin{equation*}
    k(\eta):=2a\ln
    \frac{\eta-\sqrt{\eta^2+4a\eta}}{\eta+\sqrt{\eta^2+4a\eta}}-\sqrt{\eta^2+4a\eta}.
\end{equation*}
Then
\begin{equation*}
  g_+(\eta)=g_-(\eta)= k(\eta).
\end{equation*}
Let $\xi(\eta)=\sqrt{\eta^2+4a\eta}/\eta$. It is easily checked that
$\xi(\eta)$ is a monotonically increasing function in
$-\infty<\eta\leq-4a$ with range $(-1,0]$. If
\begin{equation*}
    \widetilde{k}(\xi):=2a\ln\frac{1-\xi}{1+\xi}+4a\frac{\xi}{1-\xi^2},
\end{equation*}
then,
\begin{equation*}
    \widetilde{k}'(\xi)=\frac{8\xi^2a}{(1-\xi^2)^2}>0,
\end{equation*}
i.e., $\widetilde{k}(\xi)$ is monotonically increasing in
$-1<\xi\leq 0$. Since $k(\eta)=\widetilde{k}(\xi(\eta))$, we
conclude that $k(\eta)$ is monotonically increasing in
$-\infty<\eta\leq-4a$
with range $(-\infty,0]$. Note that $k(\eta)=0$ when $\eta=-4a$, i.e., $a=a_-$.\\

In case (ii), we have $\eta^2+4a\eta<0$. Hence,
\begin{equation*}
    u_\pm=\frac{\eta\pm i\sqrt{-4a\eta-\eta^2}}{2\eta}.
\end{equation*}
From (\ref{psi function}), it is clear that $\text{Re
}\psi(u_+)=\text{Re }\psi(u_-)$ and $\text{Im }\psi(u_+)=-\text{Im
}\psi(u_-)$. Thus, $g(\eta)$ is purely imaginary. Since $\eta$ is
negative, we have
\begin{eqnarray*}
    g(\eta)&=& i
    2\text{Im }\left[a\ln\frac{\eta-i\sqrt{-\eta^2-4a\eta}}{2\eta}-a\ln\frac{-\eta-i\sqrt{-\eta^2-4a\eta}}{2\eta}\right]-i\sqrt{-\eta^2-4a\eta}\\
    &=& i \left[2a
           \left(\arctan\frac{\sqrt{-\eta^2-4a\eta}}{-\eta}-\pi+\arctan\frac{\sqrt{-\eta^2-4a\eta}}{-\eta}\right)-\sqrt{-\eta^2-4a\eta}\right].
\end{eqnarray*}
Since $-ig'(\eta)=-\sqrt{-\eta^2-4a\eta}/\eta>0$, the function
$-ig(\eta)$ is monotonically increasing. Furthermore, since
$g(-4a)=-2\pi a i$ and $g(0)=0$, the range of this function is
$(-2a\pi,0)$.
\end{proof}

\begin{thm}\label{thm 1}
Let $0\leq a\leq \frac{1}{2}$, (i) When $a_-<a\leq \frac{1}{2}$,
there exist unique real numbers $\eta$ and $\gamma$ satisfying the
system of equations $u(w_+)=u_-$ and $u(w_-)=u_+$ in (\ref{the
relation between w+- and u+-}), i.e,
\begin{equation}\label{t+- and w+- and mapping in proof}
\begin{split}
  b\ln(1-t_+)+(1-b)&\ln t_++a \ln w_+-a\ln(w_+-1)+b\ln[1-(1-t_+)w_+] \\
   =& a \ln u_--a\ln(u_--1)+\eta u_-+\gamma, \\
  b\ln(1-t_-)+(1-b)&\ln t_-+a \ln w_--a\ln(w_--1)+b\ln[1-(1-t_-)w_-] \\
   =& a \ln u_+-a\ln(u_+-1)+\eta u_++\gamma;
\end{split}
\end{equation}
see also (\ref{phase function f for x>0}) and (\ref{mapping w to
u}). (ii) When $0\leq a\leq a_-$, there exist unique real numbers
$\eta$ and $\gamma$ satisfying the system of equations in
(\ref{correspondence on upper edge}) (since (\ref{correspondence on
lower edge}) is equivalent to (\ref{correspondence on upper edge})),
i.e,
\begin{equation}\label{t+- and w+- and mapping in proof II}
\begin{split}
  b\ln(1-t_+)+(1-b)&\ln t_++a \ln w_+^{(+)}-a\ln(w_+^{(+)}-1)+b\ln[1-(1-t_+)w_+^{(+)}] \\
   =& a \ln u_-^{(+)}-a\ln(u_-^{(+)}-1)+\eta u_-^{(+)}+\gamma, \\
  b\ln(1-t_-)+(1-b)&\ln t_-+a \ln w_-^{(+)}-a\ln(w_-^{(+)}-1)+b\ln[1-(1-t_-)w_-^{(+)}] \\
   =& a \ln u_+^{(+)}-a\ln(u_+^{(+)}-1)+\eta u_+^{(+)}+\gamma;
\end{split}
\end{equation}
\end{thm}

\begin{proof}
When $0\leq a\leq a_-$, we denote the left-hand sides and the
right-hand sides of the two equations in (\ref{t+- and w+- and
mapping in proof II}) by $L_1$, $L_2$ and $R_1$, $R_2$,
respectively. We recall that both $w_\pm^{(+)}$ are real and lie in
the interval $[0,1]$, and that they move along the upper edge of the
real line; see Figure \ref{w-sdp2}. Thus, in $L_1$ and $L_2$, we
have
\begin{equation*}
    a \ln(w_\pm^{(+)}-1)=a\ln(1-w_\pm^{(+)})+ a\pi i.
\end{equation*}
In this case, since $0<t_\pm<1$ and $0<w_\pm^{(+)}<1$, it can easily
be verified that $1-(1-t_\pm)w_\pm^{(+)}>0$. Hence, the terms in
$L_1$ and $L_2$ are well defined. Furthermore, since all terms
containing $t_\pm$ are real, we have $L_1-L_2=\text{Re
}{L_1}-\text{Re }{L_2}$. Also, for $0<w<1$,
\begin{eqnarray*}
  \text{Re } f(t_0(w),w) &=& b\ln(1-t_0(w))+(1-b)\ln t_0(w)+a\ln w \\
   &&-a\ln(1-w) +b\ln[1-(1-t_0(w))w],
\end{eqnarray*}
which is a real-valued function. It can be shown that $\text{Re }
f(t_0(w),w)$ attains its maximum value at $w=w_-$. Hence, $\text{Re
} f(t_-,w_-)>\text{Re } f(t_+,w_+)$. and there exists a non-positive
number $r(a)$ such that $L_1-L_2=r(a)$. By Lemma \ref{lem 1}(i),
$g_+(\eta)=k(\eta)$, where $k(\eta)$ is a real-valued monotonically
increasing function with range $(-\infty,0]$. Therefore, for each
$a\in[0,a_-]$, there must exist a unique $\eta$ such that
$k(\eta)=r(a)$ and the equations in (\ref{t+- and w+- and mapping in
proof II}) hold. The value of $\gamma$ is then determined
by either one of the two equations in (\ref{t+- and w+- and mapping in proof II}).\\

When $a_-<a\leq\frac{1}{2}$, we also denote the left-hand sides and
right-hand sides of the two equations in (\ref{t+- and w+- and
mapping in proof}) by $L_1$, $L_2$ and $R_1$, $R_2$, respectively.
Recall that $w_\pm$ are complex conjugate numbers. Since
$\overline{f(t,w)}=f(\overline{t},\overline{w})$, from (\ref{t+- and
w+- and mapping in proof}) we have $L_1-L_2=q(a) i$, where
\begin{eqnarray*}
    q(a)&=&2\text{Im }\{b\ln(1-t_+)+(1-b)\ln t_++a\ln
    w_+\\
    &&-a\ln(w_+-1)+b\ln[1-(1-t_+)w_+]\}.
\end{eqnarray*}
Straightforward calculation gives
\begin{equation*}
\begin{split}
    q(a)=2\biggr[&2b \arctan\frac{-\sqrt{4a-4a^2-b^2}}{2-2a+b}+(1-b)\arctan\frac{b\sqrt{4a-4a^2-b^2}}{2-2a-b^2}\\
    & \left.+2a\arctan\frac{\sqrt{4a-4a^2-b^2}}{b}-a\pi \right].
\end{split}
\end{equation*}
Note that in the present case, $4a-4a^2-b^2\geq 0$. Also, we have
\begin{equation*}
    q'(a)=-2\pi+4\arctan\frac{\sqrt{4a-4a^2-b^2}}{b}<0.
\end{equation*}
Thus, $q(a)<q(a_-)=-2\pi a_-<0$. Now, consider the function
$q(a)+2\pi a$. Since
\begin{equation*}
    \frac{\mathrm{d}}{\mathrm{d}a}\left(q(a)+2\pi a\right)=4\arctan\frac{\sqrt{4a-4a^2-b^2}}{b}>0,
\end{equation*}
it follows that $q(a)+2\pi a>q(a_-)+2\pi a_-=0$ and hence
$q(a)>-2\pi a$. Coupling these two results, we obtain $-2\pi
a<q(a)<0$. From the right-hand sides of the two equations in
(\ref{t+- and w+- and mapping in proof}), by Lemma \ref{lem 1}(ii)
we have $R_1-R_2=r(\eta)i$, where $r(\eta)$ is a real-valued
increasing function with range $(-2a\pi,0)$. Therefore, for every
fix $a\in(a_-,\frac{1}{2}]$, there exists a unique $\eta$ such that
$r(\eta)=q(a)$ and the system of equations in (\ref{t+- and w+- and
mapping in proof}) holds. The value of $\gamma$ is again determined
by either one of the two equations in (\ref{t+- and w+- and mapping
in proof}).
\end{proof}

\begin{thm}\label{thm 2}
The mapping $w\rightarrow u$ defined in (\ref{mapping w to u}) is
one-to-one and analytic when $0\leq a\leq\frac{1}{2}$.
\end{thm}
\begin{proof}
As in the cases of Charlier polynomials \cite{Rui} and Meixner
polynomials \cite{xiaojin}, we introduce an intermediate variable
$Z$ defined by
\begin{equation}\label{big Z in proof}
\begin{split}
    b\ln(1-t_0)+&(1-b)\ln t_0+a \ln w-a\ln(w-1)+b\ln[1-(1-t_0)w]\\
    =&Z=a\ln u-a\ln(u-1)+\eta
    u+\gamma,
    \end{split}
\end{equation}
where $t_0=t_0(w)$ is the relavant saddle point $t_0^+(w)$ given in
(\ref{saddle points for t}). To establish the properties of the
mapping $w\rightarrow u$, we again divide our discussion into two
cases: $0\leq a<a_-$ and $a_-<a\leq\frac{1}{2}$.\\

\begin{figure}[h]
\centering
  \includegraphics[scale=0.65]{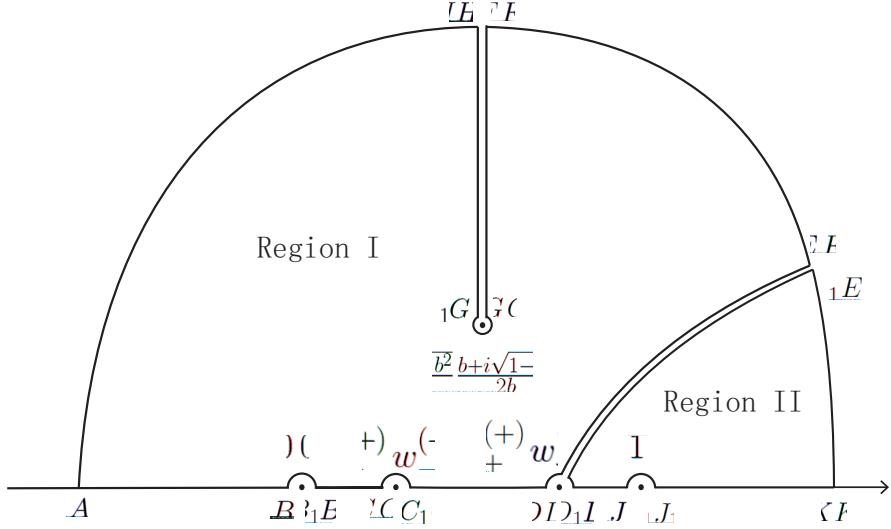}\\
  \caption{The upper half of the $w$-plane.}\label{w-proof}
\end{figure}

When $0\leq a<a_-$, we consider the upper half of the $w$-plane.
Since the functions $f(t_0(w),w)$ and $\psi(u)+\gamma$ are both
symmetric with respect to the real line, the case for the lower half
of the $w$-plane can be handled in the same manner. To avoid the
multi-valuedness around the saddle point, we divide the upper half
of the $w$-plane into two parts by using the steepest path through
$w_+^{(+)}$; see Figure \ref{w-proof}. Call these parts region I and
region II. Note that there are branch cuts, one along the infinite
interval $(-\infty,1]$ and the other along the vertical line
Re$w=\frac{1}{2}$ from $(b+i\sqrt{1-b^2})/2b$ to
$\frac{1}{2}+i\infty$. The cut along the vertical line was
introduced in (\ref{saddle points for t}).\\

As $w$ traverses along the boundary $ABB_1CC_1DD_1EFGG_1HA$ of
region I, the image point $Z$ traverses along the corresponding
boundary of a region in the $Z$-plane; see Figure \ref{z-proof-I}.
In Figure \ref{u-proof}, we draw the boundary of the region,
corresponding to region I, in the $u$-plane. The image point
$Z=\psi(u)+\gamma$ traverses along the boundary of the same region
shown in Figure \ref{z-proof-I}.

\begin{figure}[h]
\centering
  \includegraphics[scale=0.65]{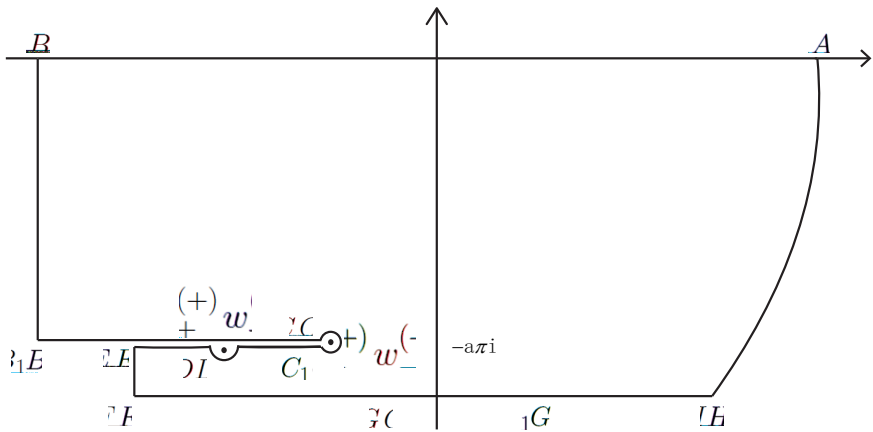}\\
  \caption{The image of Region I in the $Z$-plane.}\label{z-proof-I}
\end{figure}

\begin{figure}[h]
\centering
  \includegraphics[scale=0.65]{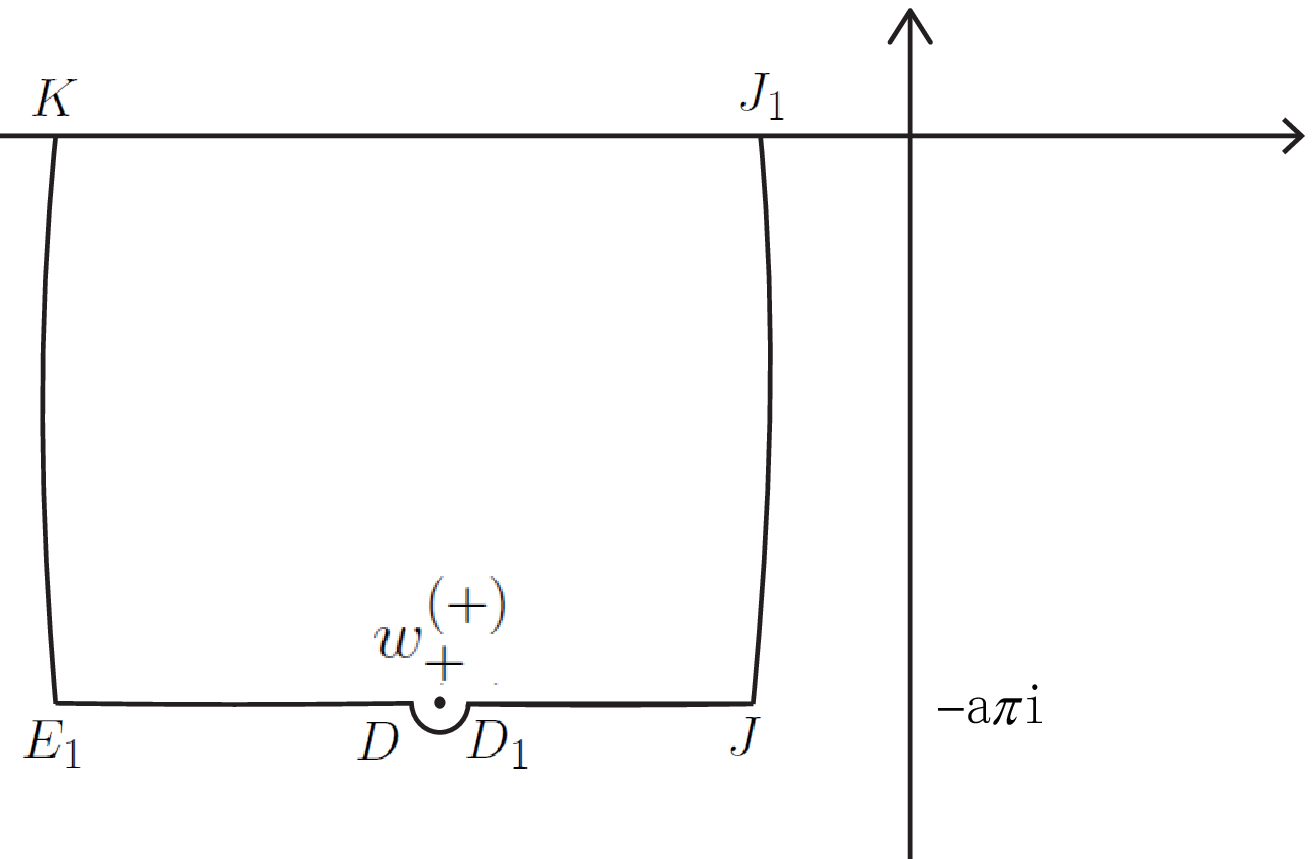}\\
  \caption{The image of Region II in the $Z$-plane.}\label{z-proof-II}
\end{figure}

\begin{figure}[h]
\centering
  \includegraphics[scale=0.65]{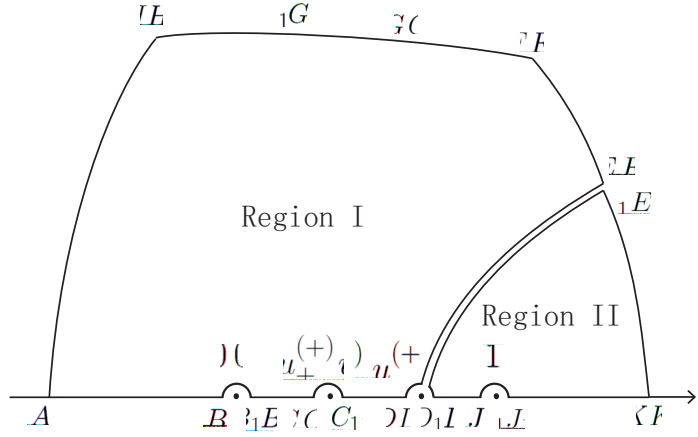}\\
  \caption{The upper half of the $u$-plane.}\label{u-proof}
\end{figure}
Let $\varphi(w)=f(t_0(w),w)$ so that (\ref{big Z in proof}) becomes
$\varphi(w)=\psi(u)+\gamma$. Since the mapping $w\rightarrow u$ is
just the composite function $\varphi^{-1}$ and $\psi(u)+\gamma$, it
is one-to-one on the boundary of region I. By the same argument, one
can prove that this mapping is also one-to-one on the boundary of
region II. For the image of region II in the intermediate $Z$-plane,
see Figure \ref{z-proof-II}. By Theorem 1.2.2 of \cite[p.12]{lecture
notes}, the mapping $w\rightarrow u$ is one-to-one in the interior
of both regions. As explained earlier, the one-to-one property of
this mapping in the lower half of the $w$-plane can be established
by using symmetry of the functions with respect to the real axis.
Note that the only possible singular points of the mapping
$w\rightarrow u$ are at $w=0$, $1$ and $w_\pm^{(+)}$. Since the
images of these points in the $u$-plane are bounded (see the points
$0$,$1$ and $u_\pm^{(+)}$ in Figure \ref{u-proof}), the mapping is
indeed one-to-one and analytic in
the whole $w$-plane.\\

When $a_-<a\leq\frac{1}{2}$, the one-to-one and analytic properties
of the mapping $w\rightarrow u$ can be proved in the same manner as
in the previous case. Hence, we simply skip this part of the argument.\\


Thus far, we have proved the one-to-one and analytic nature of the
mapping $w\rightarrow u$ defined in (\ref{mapping w to u}) when
$0\leq a<a_-$ and when $a_-<a\leq\frac{1}{2}$. To show that it also
has these properties at $a=a_-$, we note that equation (\ref{mapping
w to u}) can be written as $g(u,\eta,\gamma,f)=0$, where $g$ is just
the difference between the left-hand side and the right-hand side of
the equation. Since $\eta$, $\gamma$ and $f(t_0(w),w)$ are all
continuous in $a$ (see the proof of Theorem 1), by the implicit
function theorem $u$ is continuous in $a$ except at $u_\pm$, where
$\partial g/\partial u$ vanishes. The continuity of $u$ in $a$ when
$u=u_\pm$ comes from (\ref{saddle points of u}) and the continuity
of $\eta$. Therefore, the mapping is also one-to-one and analytic
when $a=a_-$.
\end{proof}

\section{ANALYTICITY OF $h(u,\tau)$}\label{sec 5}
Before deriving the asymptotic expansions of the integrals in
(\ref{IP for x>0 deformed}) and (\ref{IR after mapping in x<0}), we
first investigate the analyticity of the amplitude function
$h(u,\tau)$ in the neighborhood of the steepest descent surface
$S_{\tau\times u}=(-\infty,\infty)\times\gamma_u$; cf. the
statements following (\ref{two dimensional mapping}). For an
expression of this function, see (\ref{h function}) for the case
$0\leq a\leq\frac{1}{2}$, and (\ref{h function when a<0}) for the
case
$a<0$.\\

\textbf{Case 1}: $0\leq a\leq\frac{1}{2}$. \\

For $w\neq w_\pm$ and $t\neq t^+_0(w)$ (i.e., $\tau\neq 0$), we have
from (\ref{h function}), (\ref{dw over du, normal}) and (\ref{dt
over dtau})
\begin{equation}
\begin{split}
    h(u,\tau)=&\frac{u-1}{w-1}\frac{2\tau
    t(1-t)[1-(1-t)w]}{(1+b)w(t-t^+_0(w))(t-t^-_0(w))}\\
& \times
\frac{\eta(u-u_+)(u-u_-)w(1-w)\left[(2a-1)-\sqrt{1+4b^2w^2-4b^2w}\right]}{-2b^2u(u-1)(w-w_+)(w-w_-)}\label{temp
in analycity of h 1}.
\end{split}
\end{equation}
Since $\rho_t\times\gamma_w$ in the $t\times w$-plane does not pass
through $t^-_0(w)$ (unless $t^-_0(w)$ coalesces with $t^+_0(w)$, a
case which we will discuss next), we may substitute $(\ref{saddle
points for t})$ in (\ref{temp in analycity of h 1}) to give
\begin{eqnarray*}
    h(u,\tau)&=&\frac{2w\tau t(1-t)[1-(1-t)w]}{[2(1+b)wt-2w+1+\sqrt{1+4b^2(w-1)w}](t-t^+_0(w))}\\
    &&
    \times\frac{\eta(u-u_+)(u-u_-)w\left[(2a-1)-\sqrt{1+4b^2w^2-4b^2w}\right]}{b^2u(w-w_+)(w-w_-)}.
\end{eqnarray*}
Thus, in the neighborhood of the steepest descent surface
$S_{\tau\times u}$, $h(\tau,u)$ is analytic except possibly at
points $u=0$ and $u=1$, which correspond to $w=0$ and $w=1$,
respectively. From (\ref{saddle points for t}), it is easily
verified that $t^+_0(w)\rightarrow 1-b$ when $w\rightarrow 0$ and
$t^+_0(w)\rightarrow\frac{1}{1+b}$ when $w\rightarrow1$.
Substituting these values into (\ref{mapping w to u}) or (\ref{two
dimensional mapping}) yields $\frac{w}{u}=O(1)$ as $w\rightarrow
0$ and $\frac{w-1}{u-1}=O(1)$ as $w\rightarrow 1$. Hence, neither $u=0$ nor $u=1$ is the singular point in this case.\\

If $w=w_\pm$ but $t\neq t^+_0(w)$, then in (\ref{h function}) we use
(\ref{dw over du, saddle point}) instead of (\ref{dw over du,
normal}) to arrive at
\begin{eqnarray*}
  h(u,\tau) &=& \frac{u_\mp-1}{w_\pm-1}\frac{2\tau t(1-t)[1-(1-t)w_\pm]}{(1+b)w_\pm(t-t_\pm)(t-t^-_0(w_\pm))} \\
   & &
   \times\left\{\frac{\sqrt{\eta^2+4a\eta}(1-a)(1-2a)(-\eta)}{b^3\sqrt{b^2-4a+4a^2}}\right\}^{1/2}.
\end{eqnarray*}
Since we are in the case $0\leq a\leq\frac{1}{2}$, we have $w_-\neq
1$ (i.e., $u_+\neq 1$), $w_+\neq 0$ (i.e., $u_-\neq 0$) and $a\neq
\frac{1}{2}(1+\sqrt{1-b^2})$. Thus, $u_+=0$, $u_-=1$, and
$a=\frac{1-\sqrt{1-b^2}}{2}$ are the only possible singularities,
which by $(\ref{the correspondence of subcases})$ and $(\ref{the
relation between w+- and u+-})$ correspond to $w_-=0$, $w_+=1$ and
$\eta=-4a$, respectively. It can be also verified that
$\frac{w_-}{u_+}=O(1)$ as $u_+\rightarrow 0$,
$\frac{w_+-1}{u_--1}=O(1)$ as $u_-\rightarrow 1$, and
${\sqrt{\eta^2+4a\eta}}/{\sqrt{b^2-4a+4a^2}}=O(1)$
as $a\rightarrow\frac{1}{2}(1-\sqrt{1-b^2})$. Hence, these possible singularities are actually removable.\\

If $w\neq w_\pm$ but $t=t^+_0(w)$, then in (\ref{h function}) we use
the second equation (instead of the first one) in (\ref{dt over
dtau}) to reach
\begin{eqnarray*}
  h(u,\tau)&=&\left\{\frac{2t^+_0(w)(1-t^+_0(w))[1-(1-t^+_0(w))w]}{\sqrt{1+4b^2(w-1)w}}\right\}^{1/2}\\
    &&
    \times\frac{\eta(u-u_+)(u-u_-)w\left[(2a-1)-\sqrt{1+4b^2w^2-4b^2w}\right]}{2b^2u(w-w_+)(w-w_-)},
\end{eqnarray*}
where, in addition to the possible singular point at $u=0$ which we
have just discussed, we have introduced two new possible
singularities $w={(b\pm i\sqrt{1-b^2})}/{2b}$. However, from Figure
\ref{w+-} it is clear that these two points will not be in the
neighborhood of $\rho_t\times\gamma_w$, unless $a\rightarrow
\frac{1}{2}$ in which case $w\rightarrow
w_\pm$ (see the discussion below).\\

If $w=w_\pm$ and $t=t^+_0(w)$, then in (\ref{h function}) we use
(\ref{dw over du, saddle point}) and the second equation in (\ref{dt
over dtau}). Since $t^+_0(w_\pm)=t_\pm$, the result is
\begin{equation}\label{h function when a>0 in analyticity section}
\begin{split}
  h(u,\tau)=&\frac{u_\mp-1}{w_\pm-1}\left\{\frac{2t_\pm(1-t_\pm)[1-(1-t_\pm)w_\pm]}{b^3}\right.\\
  &\left. \times\frac{\sqrt{\eta^2+4a\eta}(1-a)(1-2a)(-\eta)}{\sqrt{b^2-4a+4a^2}}\right\}^{1/2}.
\end{split}
\end{equation}
Note that $a=\frac{1}{2}$ is not a singular point, and since
${\sqrt{\eta^2+4a\eta}}/{\sqrt{b^2-4a+4a^2}}=O(1)$
as $a\rightarrow\frac{1}{2}(1-\sqrt{1-b^2})$ as we have discussed before, there is no singular point in the neighborhood of the integration surface.\\

\textbf{Case 2}: $a<0$. \\

In this case, only the negative saddle point $w_-$ is relevant, and
the steepest descent path associated with $w_-$ will not pass
through $w_+$. Thus, the situation is much simpler than the case
$0\leq a\leq\frac{1}{2}$. By using (\ref{h function when a<0}),
(\ref{dt over dtau}), (\ref{dw over du when a<0}) and (\ref{dw over
du when a<0 saddle point}), we have the following analytic
expressions for
$h(u,\tau)$:\\

If $w\neq w_-$ and $t\neq t^+_0(w)$, we have
\begin{equation*}\label{}
    h(u,\tau)=\frac{(u-a)\left[(1-2a)+\sqrt{1+4b^2w^2-4b^2w}\right]}{b^2(w-w_+)(w-w_-)}\frac{\tau
    t(1-t)[1-(1-t)w]}{(1+b)(t-t^+_0(w))(t-t^-_0(w))}.
\end{equation*}
Since $w\neq w_-$ and $t\neq t^+_0(w)$, clearly there is no
singularity in the neighborhood of the integration surface.\\

If $w\neq w_-$ but $t= t^+_0(w)$, then
\begin{equation*}\label{}
    h(u,\tau)=\left\{\frac{2t^+_0(w)(1-t^+_0(w))[1-(1-t^+_0(w))w]}{\sqrt{1+4b^2(w-1)w}}\right\}^{1/2}\frac{(u-a)w\left[(1-2a)+\sqrt{1+4b^2w^2-4b^2w}\right]}{2b^2(w-w_+)(w-w_-)}.
\end{equation*}
Again, there is no singularity.\\

If $w= w_-$ (in which case $u=a$ by (\ref{relation betwen w and u
when a<0})) but $t\neq t^+_0(w)$, we have
\begin{equation*}\label{}
    h(u,\tau)=\frac{-2\tau
    t(1-t)[1-(1-t)w_-]}{(1+b)(t-t_-)(t-t_0^-(w_-))}\left\{\frac{ab(1-2a)}{(1-a)\sqrt{b^2-4a-4a^2}}\right\}^{1/2}.
\end{equation*}
The only possible singularity occurs when $w_-=0$ (i.e., when
$a\rightarrow 0^-$). Since it can be shown that $\frac{a}{w_-}=O(1)$
as $a\rightarrow 0^-$, $h(u,\tau)$ is again bounded and analytic as $a\rightarrow 0^-$.\\

If $w=w_-$ and $t=t^+_0(w)$, then
\begin{equation}\label{h_0 a<0 analycity}
\begin{split}
     h(u,\tau)=&\left\{\frac{2aw_-t_-(1-t_-)[1-(1-t_-)w_-]}{(1-w_-)b\sqrt{b^2-4a-4a^2}}\right\}^{1/2}
\end{split}
\end{equation}
and there is no possible singularity at all.\\

The conclusions in all these cases infer that for $a<0$, $h(u,\tau)$
is analytic in the neighborhood of the surface of integration.

\section{ASYMPTOTIC EXPANSIONS}\label{sec 6}
\textbf{Case 1}: $0\leq a\leq\frac{1}{2}$.\\

Let $h_0(u,\tau)\equiv h(u,\tau)$, where $h(u,\tau)$ is the
amplitude function given in (\ref{h function}). Write
\begin{equation}\label{h0}
     h_0(u,\tau) = a_0(\tau)+b_0(\tau)u-(u-u_+)(u-u_-)g_0(u,\tau).
\end{equation}
Clearly,
\begin{equation}\label{a0 and b0}
    \begin{split}
       a_0(\tau)&=\frac{u_-h_0(u_+,\tau)-u_+h_0(u_-,\tau)}{u_--u_+},\\
       b_0(\tau)&=\frac{h_0(u_-,\tau)-h_0(u_+,\tau)}{u_--u_+},
     \end{split}
\end{equation}
and
\begin{eqnarray}
  g_0(u,\tau) = \frac{1}{(u-u_+)(u-u_-)}\left[h_0(u,\tau)-a_0(\tau)-b_0(\tau)u\right].\label{g0}
\end{eqnarray}
For $u\neq u_\pm$, $g_0(u,\tau)$ is analytic in $u$ wherever
$h_0(u,\tau)$ is analytic. Furthermore, it is easily shown that the
limits
\begin{eqnarray*}
  \lim_{u\rightarrow u_+}g_0(u,\tau) &=& \frac{h'_0(u_+,\tau)}{u_+-u_-}-\frac{h_0(u_+,\tau)-h_0(u_-,\tau)}{(u_+-u_-)^2}\quad  \text{if}\quad a\neq a_-, \\
  \lim_{u\rightarrow u_-}g_0(u,\tau) &=& \frac{h'_0(u_-,\tau)}{u_--u_+}+\frac{h_0(u_+,\tau)-h_0(u_-,\tau)}{(u_+-u_-)^2}\quad  \text{if}\quad a\neq a_-, \\
  \lim_{u\rightarrow u_+}g_0(u,\tau) &=& \lim_{u\rightarrow u_-}g_0(u,\tau)=\frac{1}{2}h^{''}_0(u_-,\tau)\quad \quad \quad\quad \text{if}\quad a=
  a_-,
\end{eqnarray*}
exist and are finite. Thus, $g_0(u,\tau)$ is actually analytic in $u$ everywhere $h_0(u,\tau)$ is analytic.\\

Recall the integral representation of the confluent hypergeometric
function $\mathbf{M}(aN+1,1,\eta N)$ in (\ref{confluent
hypergeometric function for a>0})
\begin{equation}\label{M function in expansion a>0}
    \mathbf{M}(aN+1,1,\eta N)=\frac{1}{2\pi
    i}\int_{\gamma_1}e^{N\psi(u)}\frac{\mathrm{d}u}{u-1},
\end{equation}
where $\psi(u)=a\ln u-a\ln(u-1)+\eta u$ and $\gamma_1$ is the loop
contour which starts at $u=0$, encircles $u=1$ in the positive
direction, and returns to $u=0$; see (\ref{integral temp 3}). From
(\ref{m function a>0}), we also have
\begin{equation}\label{M' function in expansion a>0}
    \mathbf{M}'(aN+1,1,\eta N)=\frac{1}{2\pi
    i}\int_{\gamma_1}\frac{u}{u-1}e^{N\psi(u)}\mathrm{d}u,
\end{equation}
where the derivative is taken with respect to the third variable in
the function $\mathbf{M}(c,d,z)$. Inserting (\ref{h0}) in (\ref{IP
for x>0 deformed}), we obtain
\begin{equation}
\begin{split}
  t_n(x,N+1) = \frac{(-1)^n\Gamma(n+N+2)}{\Gamma(n+1)\Gamma(N-n+1)}e^{N\gamma}&\left[ \mathbf{M}(aN+1,1,\eta N)\int_{-\infty}^{+\infty}a_0(\tau)e^{-N\tau^2}\mathrm{d}\tau\right.\\
   & \left. +\mathbf{M}'(aN+1,1,\eta
   N)\int_{-\infty}^{+\infty}b_0(\tau)e^{-N\tau^2}\mathrm{d}\tau+\epsilon_1^+
   \right],\label{expansion of the first term}
   \end{split}
\end{equation}
where
\begin{equation*}
    \epsilon_1^+ =-\frac{1}{2 \pi i}\int_{-\infty}^{+\infty}\left[\int_{\gamma_u}(u-u_+)(u-u_-)g_0(u,\tau)e^{N
\psi(u)}\frac{\mathrm{d}u}{u-1}\right]e^{-N\tau^2}\mathrm{d}\tau.
\end{equation*}
To the inner integral, we apply an integration-by-parts; the result
is that it is equal to
\begin{equation*}
\frac{1}{\eta N}u e^{N
\psi(u)}g_0(u,\tau)\biggr|_{u=0^-}^{u=0^+}-\frac{1}{\eta
N}\int_{\gamma_u}h_1(u,\tau)e^{N \psi(u)}\frac{\mathrm{d}u}{u-1},
\end{equation*}
where
\begin{equation*}
h_1(u,\tau)=(u-1)\left[g_0(u,\tau)+u\frac{\partial
g_0(u,\tau)}{\partial u}\right].
\end{equation*}
By using (\ref{g0}) and the fact that $h_0(u)=O(1)$ as $u\rightarrow
0$ as we have shown in $\S$\ref{sec 5}, the first term above is
actually zero. Thus,
\begin{equation}\label{epsilon1}
    \epsilon_1^+= \frac{1}{2\pi i \eta N}\int_{-\infty}^{+\infty}\int_{\gamma_u}h_1(u,\tau)\frac{e^{N
\psi(u)}}{u-1}e^{-N\tau^2}\mathrm{d}u\mathrm{d}\tau.
\end{equation}
The two integrals in (\ref{expansion of the first term}) can be
asymptotically evaluated by using Watson's lemma
\cite[p.20]{R.Wong}. So, let us expand $a_0(\tau)$ and $b_0(\tau)$
into Maclaurin expansions
\begin{equation*}\label{a0 and b0 zhankai}
    a_0(\tau)=\sum_{j=0}^{\infty} a_{0,j}\tau^j,\hspace{2cm} b_0(\tau)=\sum_{j=0}^{\infty} b_{0,j}\tau^j.
\end{equation*}
Since
\begin{equation}\label{integral of tau}
    \int_{-\infty}^{+\infty}\tau^{2j}e^{-N\tau^2}\mathrm{d}\tau
      =
\frac{\Gamma(j+\frac{1}{2})}{N^{-j-\frac{1}{2}}},\hspace{2cm}\int_{-\infty}^{+\infty}\tau^{2j+1}e^{-N\tau^2}\mathrm{d}\tau
       =0,
\end{equation}
we obtain
\begin{equation}\label{}
\begin{split}
    t_n(x,N+1) = \frac{(-1)^n\Gamma(n+N+2)}{\Gamma(n+1)\Gamma(N-n+1)} e^{N\gamma}&\left\{ \mathbf{M}(a N+1,1,\eta N)  \left[\frac{c_0}{\sqrt{N}}+\delta^+_1\right] \right.\nonumber \\
   & \left. +\mathbf{M}'(a N+1,1,\eta
   N)\left[\frac{d_0}{\sqrt{N}}+\widetilde{\delta}^+_1\right]+\epsilon_1^+ \right\},
\end{split}
\end{equation}
where $\delta^+_1$ and $\widetilde{\delta}^+_1$ denote error terms
that are $O(N^{-3/2})$.\\

The above procedure can be repeated. For $l=0,1,2,\cdot\cdot\cdot$,
we define recursively
\begin{eqnarray}
  h_l(u,\tau) &=& a_l(\tau)+b_l(\tau)u-(u-u_-)(u-u_+)g_l(u,\tau),\label{h_l}\\
  h_{l+1}(u,\tau)&=&(u-1)\left[g_l(u,\tau)+u\frac{\partial
g_l(u,\tau)}{\partial u}\right],\label{hl 2}
\end{eqnarray}
and
\begin{equation}\label{expansion of coeffiencet an and bn}
       a_l(\tau)=\sum_{j=0}^{\infty} a_{l,j}\tau^j, \qquad       b_l(\tau)=\sum_{j=0}^{\infty}
       b_{l,j}\tau^j.
\end{equation}
Using induction, one can show that $h_l(u,\tau)$ and $g_l(u,\tau)$
are both analytic and bounded as $u\rightarrow 0$. Hence,
\begin{equation}\label{result a>0}
\begin{split}
  t_n(x,N+1) = \frac{(-1)^n\Gamma(n+N+2)e^{N\gamma}}{\Gamma(n+1)\Gamma(N-n+1)}&\left\{\mathbf{M}(a N+1,1,\eta N)\left[\sum_{l=0}^{p-1}\frac{c_l}{N^{l+\frac{1}{2}}}+\delta^+_p\right]\right.\\
   & \left. +\mathbf{M}'(a N+1,1,\eta N)\left[\sum_{l=0}^{p-1}\frac{d_l}{N^{l+\frac{1}{2}}}+\widetilde{\delta}_p^+\right]+\epsilon_p^+\right\},
\end{split}
\end{equation}
where
\begin{eqnarray}
  c_l &=& \sum_{m=0}^la_{l-m,2m}\Gamma(\frac{2m+1}{2}) \label{c_l}\label{cl when a>0},\\
  d_l& =& \sum_{m=0}^lb_{l-m,2m}\Gamma(\frac{2m+1}{2}) \label{d_l}\label{d_l when a>0},\\
  \epsilon_p^+&=&\frac{1}{2\pi i \eta N}\frac{1}{N^p}\int_{-\infty}^{+\infty}\int_{\gamma_u}h_p(u,\tau)\frac{e^{N
\psi(u)}}{u-1}e^{-N\tau^2}\mathrm{d}u\mathrm{d}\tau, \label{epsilonp
in bigger than 0}
\end{eqnarray}
and $\delta_p^+$, $\widetilde{\delta}_p^+$ are error terms of order $O(N^{-p-1/2})$.\\

To prove that for fixed $b=n/N\in(0,1)$, (\ref{result a>0}) is a
uniform asymptotic expansion in $a\in[0,\frac{1}{2}]$, we need to
show that there are positive constants $A_p$ and $B_p$, independent
of $a$, such that
\begin{equation}\label{error estimate when a>0}
    |\epsilon_p^+|\leq\frac{A_p}{N^{p+\frac{1}{2}}}|\mathbf{M}(a N+1,1,\eta N)|+\frac{B_p}{N^{p+\frac{1}{2}}}|\mathbf{M}'(a N+1,1,\eta
    N)|.
\end{equation}\\

\textbf{Case 2:} $a<0$.\\

Proceeding in the same manner as in Case 1, we put
$h_0(u,\tau)\equiv h(u,\tau)$, where $h(u,\tau)$ is given by (\ref{h
function when a<0}), and write
\begin{equation}
    h_0(u,\tau)=a_0(\tau)+(u-a)g_0(u,\tau).
\end{equation}
Since $\lim\limits_{u\rightarrow a}g_0(u,\tau)=h_0'(a,\tau)$,
$g_0(u,\tau)$ is analytic in $u$ wherever $h_0(u,\tau)$ is. By
(\ref{IR after mapping in x<0}) and (\ref{gamma function in use}),
we have
\begin{equation}\label{expansion of leading term for x<0}
  t_n(x,N+1) = \frac{(-1)^n\Gamma(n+N+2)N^{-a N}e^{N\gamma}}{\Gamma(n+1)\Gamma(N-n+1)\Gamma(-aN+1)}\left[\int_{-\infty}^{+\infty}a_0(\tau)e^{-N\tau^2}\mathrm{d}\tau+\epsilon_1^-
  \right],
\end{equation}
where
\begin{equation*}
    \epsilon_1^- =\frac{N^{aN}\Gamma(-a
  N+1)}{2\pi i}\int_{-\infty}^{+\infty}\left[\int_{+\infty}^{(0+)}(u-a)g_0(u,\tau)e^{N
\widetilde{\psi}(u)}\frac{\mathrm{d}u}{u}\right]e^{-N\tau^2}\mathrm{d}\tau
\end{equation*}
and $\widetilde{\psi}(u)$ is given in (\ref{psi function when a<0}).
To the inner integral, we apply an integration-by-parts. The result
is that it is equal to
\begin{equation*}
 -\frac{1}{N}g_0(u,\tau)e^{N
\widetilde{\psi}(u)}\biggr|^{u=\infty e^{2\pi
i}}_{u=\infty}+\frac{1}{N}\int_{+\infty}^{(0^+)}h_1(u,\tau)e^{N
\widetilde{\psi}(u)}\frac{\mathrm{d}u}{u},
\end{equation*}
where
\begin{equation*}
    h_1(u,\tau)=u\frac{\partial g_0(u,\tau)}{\partial u}.
\end{equation*}
Since $\widetilde{\psi}(u)\sim -u$ as $u\rightarrow +\infty$, the
first term vanishes and we have
\begin{equation*}\label{epsilon1 in case x<0}
    \epsilon_1^-=\frac{N^{aN}\Gamma(-a
  N+1)}{2\pi i N}\int_{-\infty}^{+\infty}\left[\int_{+\infty}^{(0^+)}h_1(u,\tau)e^{N
\widetilde{\psi}(u)}\frac{\mathrm{d}u}{u}\right]e^{-N\tau^2}\mathrm{d}\tau.
\end{equation*}
The Laplace integral in (\ref{expansion of leading term for x<0})
can again be asymptotically evaluated by using Watson's lemma, and
the result is
\begin{equation}\label{expansion of first term when a<0}
      t_n(x,N+1)=\frac{(-1)^n\Gamma(n+N+2)N^{-a N}e^{N\gamma}}{\Gamma(n+1)\Gamma(N-n+1)\Gamma(-a
  N+1)}\left[\frac{a_{0,0}}{\sqrt{N}}\Gamma(\frac{1}{2})+\epsilon_1^-+\delta^-_1\right],
\end{equation}
where $a_{0,0}$ is the constant term in the expansion of $a_0(\tau)$
and $\delta_1^-$ denotes the error term resulting from Watson's
lemma; see (\ref{integral of tau}).\\

The process can be repeated. For $l=0,1,2,\cdot\cdot\cdot$, we
define
\begin{subequations}
\begin{align}
    h_l(u,\tau) &= a_l(\tau)+(u-a)g_l(u,\tau),\label{6.21 a}\\
h_{l+1}(u,\tau) &= u\frac{\partial g_{l}(u,\tau)}{\partial u}, \label{hl when a<0}\\
  a_l(\tau)&=\sum_{j=0}^{+\infty} a_{l,j}\tau^j.\label{al zhankai in
  a<0}
  \end{align}
\end{subequations}
As before, we arrive at
\begin{equation}\label{result a<0}
      t_n(x,N+1)=\frac{(-1)^n\Gamma(n+N+2)N^{-a N}e^{N\gamma}}{\Gamma(n+1)\Gamma(N-n+1)\Gamma(-a
  N+1)}\left[\sum_{l=0}^{p-1}\frac{c_l}{N^{l+\frac{1}{2}}}+\epsilon_p^-+\delta^-_p\right],
\end{equation}
where
\begin{equation}\label{cl when a<0}
    c_l = \sum_{m=0}^la_{l-m,2m}\Gamma(\frac{2m+1}{2})
\end{equation}
and
\begin{equation}\label{epsilon p}
    \varepsilon_p^- =\frac{N^{aN}\Gamma(-a
  N+1)}{2\pi i N^p}\int_{-\infty}^{+\infty}\left[\int_{+\infty}^{(0^+)}h_p(u,\tau)e^{N
\widetilde{\psi}(u)}\frac{\mathrm{d}u}{u}\right]e^{-N\tau^2}\mathrm{d}\tau.
\end{equation}
The error term $\delta_p^-$ comes from applying Watson's lemma to
the functions $a_l(\tau)$, $0\leq l\leq p-1$, and is of order
$O(N^{-p-\frac{1}{2}})$. Thus, to prove that expansion (\ref{result
a<0}) is uniform with respect to $a$, we only need to show that
there exists a constant $M_p$, independent of $a$, such that
\begin{equation}\label{error estimate in x<0}
    |\epsilon_p^-|\leq\frac{M_p}{N^{p+\frac{1}{2}}}.
\end{equation}

\section{ERROR ESTIMATES}\label{sec 7}
To estimate the error terms $\epsilon_p^+$ and $\epsilon_p^-$, we
shall actually provide two slightly stronger versions of the results
stated in (\ref{error estimate when a>0}) and (\ref{error estimate
in x<0}). Indeed, we shall show that the constants $A_p$ and $B_p$
in (\ref{error estimate when a>0}) are independent of
$b\in[\delta,1-\delta]$ and $a\in[0,\frac{1}{2}]$ for any
$\delta>0$, and similarly that the constant $M_p$ in (\ref{error
estimate in x<0}) is independent of $b\in[\delta,1-\delta]$ and
$a\in(-\infty, 0)$. That is, the expansion in (\ref{result a>0}) is
uniform in $b\in[\delta,1-\delta]$ and $a\in[0,\frac{1}{2}]$, and
the expansion in (\ref{result a<0}) is uniform in
$b\in[\delta,1-\delta]$ and
$a\in(-\infty,0)$.\\

To proceed, we will divide our discussion into five cases:
\begin{enumerate}
  \item[(i)] $a_-<a\leq \frac{1}{2}$ and $N^{\frac{2}{3}}(\eta+4a)$ is
  unbounded;
  \item[(ii)] $N^{\frac{2}{3}}(\eta+4a)$ is bounded;
  \item[(iii)] $0\leq a<a_-$, $N^{\frac{2}{3}}(\eta+4a)$ is unbounded
  and $aN$ is unbounded;
  \item[(iv)] $a>0$ and $aN$ is bounded;
  \item[(v)] $a<0$.
\end{enumerate}
From (\ref{mapping w to u}), we note that $a=a_-$ corresponds
$\eta=-4a$.\\

\textbf{Case (i)} First we recall the following result of Fedoryuk
\cite[p.124-125]{Fedoryuk}; consider the integral
\begin{equation*}
    I(\lambda)=\int_\gamma f(z)e^{\lambda S(z)}\mathrm{d}z, \qquad
    z=(z_1,\cdot\cdot\cdot,z_n)\in\mathbb{C}^n,
\end{equation*}
where $\lambda>0$,
$\mathrm{d}z=\mathrm{d}z_1\cdot\cdot\cdot\mathrm{d}z_n$ and $\gamma$
is a smooth manifold of (real) dimension $n$. Assume that the
function $f(z)$ and $S(z)$ are analytic in some domain $D$
containing $\gamma$. A point $z^0$ is called a \emph{saddle point}
of $S(z)$ if $\triangledown S(z^0)=0$, and a \emph{simple} saddle
point if $\det S''(z^0)\neq 0$.

\begin{thm}\label{thm 3}
Let $\max\limits_{z\in\gamma}\text{Re } S(z)$ be attained only at a
point $z^0$, which is a simple saddle point and an interior point of
$\gamma$. Then, as $\lambda\rightarrow +\infty$, we have
\begin{equation*}
    I(\lambda)=\left(\frac{2\pi}{\lambda}\right)^{n/2}\left[\det S''(z^0)\right]^{-1/2}e^{\lambda
    S(z^0)}\left[f(z^0)+O\left(\frac{1}{\lambda}\right)\right].
\end{equation*}
The choice of the branch for the root depends on the orientation of
the integration surface.
 \end{thm}

In our present case, $n=2$ and there are two simple saddle points
$(u_+,0)$ and $(u_-,0)$. By Theorem \ref{thm 3}, it follows from
(\ref{epsilonp in bigger than 0}) that as $N\rightarrow \infty$
\begin{equation*}
\begin{split}
  \epsilon_p^+ \sim  -\frac{(\eta+4a)^{-1/4}e^{N\eta/2}}{\sqrt{2}i(-\eta)^{5/4} N^{p+2}}\biggr[&h_p(u_+,0)e^{i(aN\pi-2aN\varphi_-+\frac{1}{2} N\sqrt{-\eta^2-4a\eta}+\frac{3\pi}{4}-\varphi_-)} \nonumber\\
   & \left.+h_p(u_-,0)e^{-i(aN\pi-2aN\varphi_-+\frac{1}{2} N\sqrt{-\eta^2-4a\eta}+\frac{3\pi}{4}-\varphi_-)}\right],\label{error in case
   1}
   \end{split}
\end{equation*}
where $\varphi_-$ is the argument of $u_-$ in $(-\pi,\pi]$. The
asymptotic formulas of $\mathbf{M}(aN+1,1,\eta N)$ and
$\mathbf{M}'(aN+1,1,\eta N)$ can be obtained in the same manner with
$n=1$ and $h_p(u,\tau)$ replaced by $1$. Thus,
\begin{equation*}
\begin{split}
  \mathbf{M}(aN+1,1,\eta N) \sim \frac{(\eta+4a)^{-1/4}e^{N\eta/2}}{\sqrt{2\pi N}i(-\eta)^{1/4}}\biggr[&e^{i(aN\pi-2aN\varphi_-+\frac{1}{2} N\sqrt{-\eta^2-4a\eta}+\frac{3\pi}{4}-\varphi_-)}\\
   & \left.+e^{-i(aN\pi-2aN\varphi_-+\frac{1}{2} N\sqrt{-\eta^2-4a\eta}+\frac{3\pi}{4}-\varphi_-)}\right],\label{M in case 1}
\end{split}
\end{equation*}

\begin{equation*}
    \begin{split}
  \mathbf{M}'(aN+1,1,\eta N) \sim \frac{(\eta+4a)^{-1/4}e^{N\eta/2}\sqrt{a}}{\sqrt{2\pi N}i(-\eta)^{3/4}}\biggr[&e^{i(aN\pi-2aN\varphi_-+\frac{1}{2} N\sqrt{-\eta^2-4a\eta}+\frac{3\pi}{4}-2\varphi_-)} \nonumber\\
   & \left.+e^{-i(aN\pi-2aN\varphi_-+\frac{1}{2} N\sqrt{-\eta^2-4a\eta}+\frac{3\pi}{4}-2\varphi_-)}\right].\label{M' in case
   1}
\end{split}
\end{equation*}
A combination of the last three equations gives the error estimate
in (\ref{error estimate when a>0}).\\

\textbf{Case (ii)} Here, $N^{2/3}(\eta+4a)$ is bounded. Hence,
$\eta\rightarrow -4a$ and, correspondingly, $a\rightarrow a_-$; see
the statement preceding (\ref{mapping w to u}). For $a<a_-$, the
contour $\gamma_u$ in (\ref{epsilonp in bigger than 0}) is shown in
Figure \ref{u-sdp2}; for $a>a_-$, the contour $\gamma_u$ is depicted
in Figure \ref{u-sdp1}. We divide $\gamma_u$ into two parts, so that
the error term $\epsilon_p^+$ in (\ref{epsilonp in bigger than 0})
can be written correspondingly as $\epsilon_p^+\equiv
\epsilon_{p,1}^++\epsilon_{p,2}^+$. We call the part of $\gamma_u$
in the upper half of the plane $\gamma_u^{(1)}$, and the part of
$\gamma_u$ in the lower half of the plane $\gamma_u^{(2)}$. Because
of the cut in the interval $(0,1)$, the phase function $\psi(u)$ of
the $u$-integral in (\ref{epsilonp in bigger than 0}) can be written
as
\begin{equation*}
    \psi(u)=a\ln u-a\ln (1-u)+\eta u-ia\pi\equiv
    \widehat{\psi}(u)-ia\pi
\end{equation*}
for $u\in \gamma_u^{(1)}$, and
\begin{equation*}
    \psi(u)=\widehat{\psi}(u)+ia\pi
\end{equation*}
for $u\in \gamma_u^{(2)}$. Since $\widehat{\psi}(u)$ has two
coalescing saddle points $u_\pm$ given in (\ref{saddle points of
u}), we make the standard transformation
\begin{equation}\label{airy transformation in case 2}
    \widehat{\psi}(u)=\frac{1}{3}v^3-\zeta v+A
\end{equation}
with the correspondence of the saddle points $u_\pm$ and
$\pm\sqrt{\zeta}$ given by
\begin{equation}\label{airy transformation mapping of saddle points}
    u_-\leftrightarrow \sqrt{\zeta},\qquad \quad u_+\leftrightarrow-\sqrt{\zeta}.
\end{equation}
If $\zeta>0$, then $\pm\sqrt{\zeta}$ are both real; if $\zeta<0$,
then $\pm\sqrt{\zeta}$ are complex conjugates and purely imaginary.
The values of $A$ and $\zeta$ can be obtained by using (\ref{airy
transformation in case 2}) and (\ref{airy transformation mapping of
saddle points}). The images of the contour $\gamma_2^{(2)}$ under
the mapping $u\rightarrow v$ defined in (\ref{airy transformation in
case 2}) are depicted in Figures \ref{SDP of airy 1} and \ref{SDP of
airy 2}. The part $\epsilon^+_{p,1}$ of the error term
$\epsilon_p^+$ now becomes
\begin{equation*}
    \epsilon^+_{p,1}=\frac{e^{-ia\pi N+\frac{1}{2}\eta N}}{2\pi i\eta
    N^{p+1}}\int_{-\infty}^{\infty}\int_{C_2}g(v,\tau)e^{N(\frac{1}{3}v^3-\zeta
    v)}e^{-N\tau^2}\mathrm{d}v\mathrm{d}\tau,
\end{equation*}
where
\begin{equation}\label{g function in case 2 error estimate}
    g(v,\tau)=h_p(u,\tau)\frac{1}{u-1}\frac{\mathrm{d}u}{\mathrm{d}v}.
\end{equation}
\vspace{1.3cm}

\begin{figure}[!h]
\centering
\begin{minipage}[t]{0.4\textwidth}
\centering
  \includegraphics[width=100pt,bb=110 0 330 170]{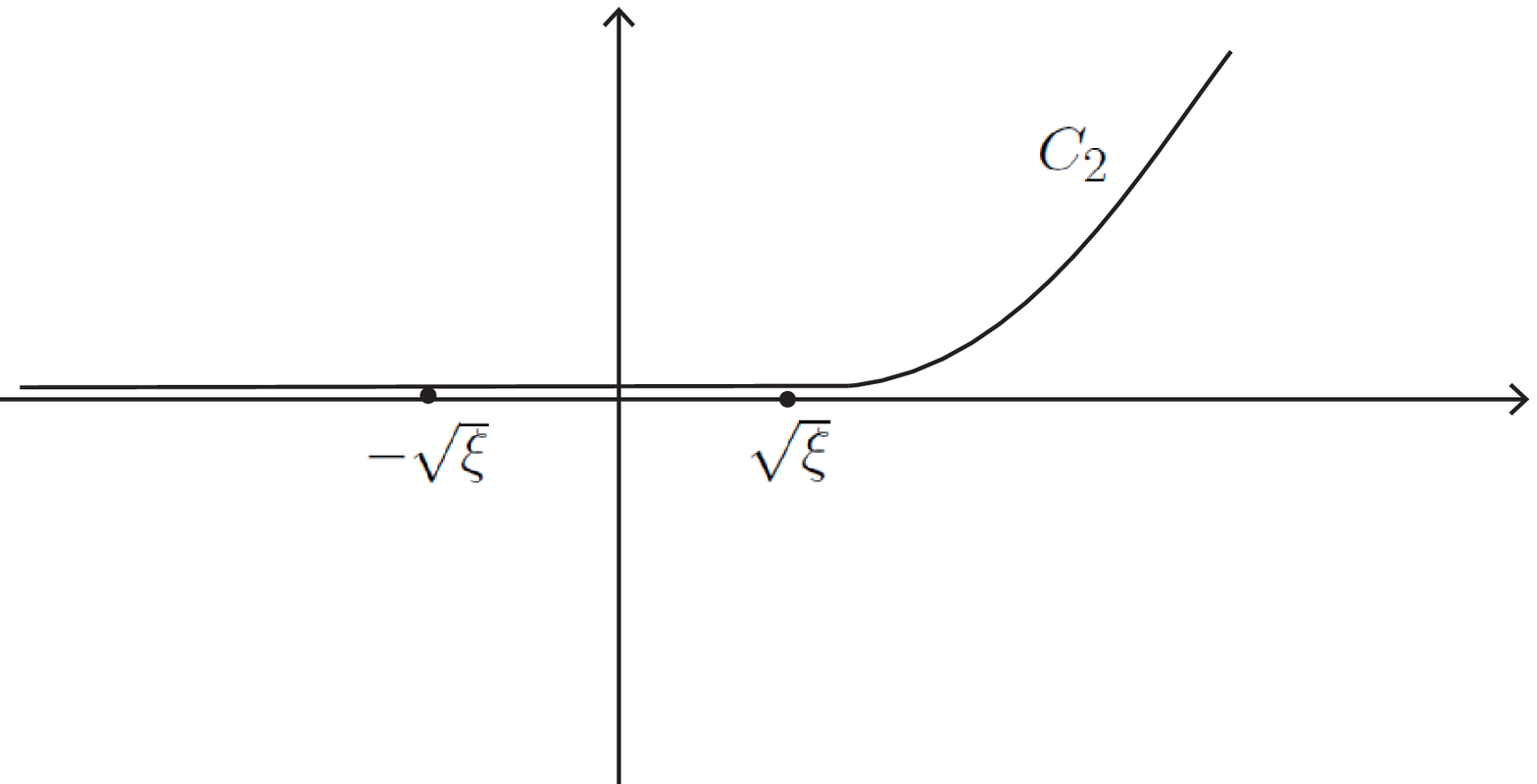}\\
  \caption{Contour $C_2$ $(a<a_-)$.}\label{SDP of airy 1}
\end{minipage}\hfill
\begin{minipage}[t]{0.45\textwidth}
\centering
  \includegraphics[width=100pt,bb=110 0 330 170]{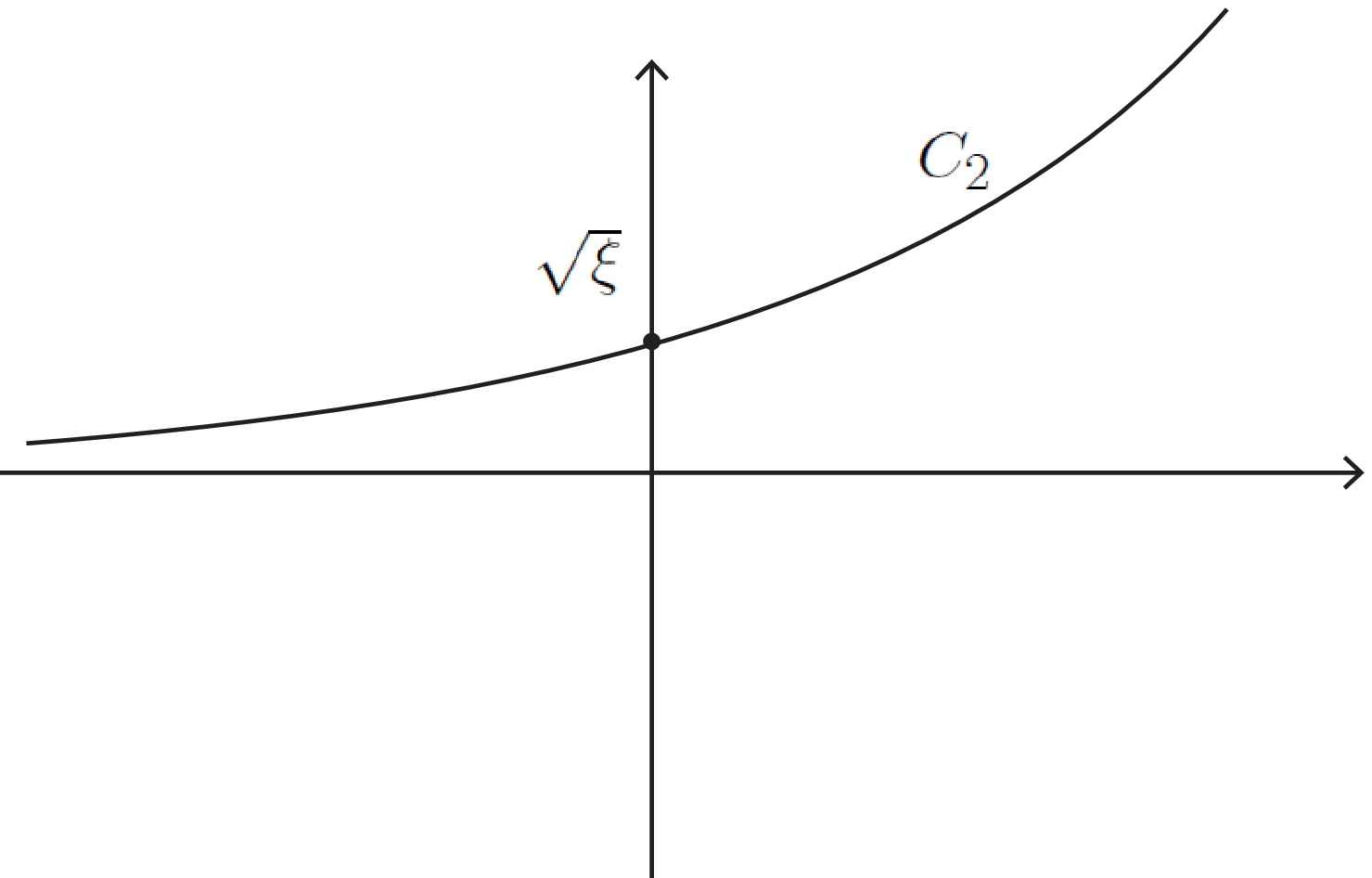}\\
  \caption{Contour $C_2$ $(a>a_-)$.}\label{SDP of airy 2}
\end{minipage}\hfill
\end{figure}

The part $\epsilon^+_{p,2}$ can be handled in exactly the same
manner, and we have
\begin{equation*}
    \epsilon^+_{p,2}=\frac{e^{ia\pi N+\frac{1}{2}\eta N}}{2\pi i\eta
    N^{p+1}}\int_{-\infty}^{\infty}\int_{C_3}g(v,\tau)e^{N(\frac{1}{3}v^3-\zeta
    v)}e^{-N\tau^2}\mathrm{d}v\mathrm{d}\tau,
\end{equation*}
where the contour $C_3$ is the reflection of $C_2$ in the $v$-plane
with respect to the real axis. Since the two contours $C_2$ and
$C_3$ together are equivalent to a contour running from $\infty
e^{-i\pi/3}$ to $\infty e^{i\pi/3}$ in the right half-plane, in
terms of Airy integral Ai$(\cdot)$ we have by the Bleistein method
\cite[p.366-370]{R.Wong}
\begin{equation}
\begin{split}
    \epsilon^+_p\sim& \text{Re }\left\{\frac{e^{-ia\pi N+\frac{1}{2}\eta N}}{\sqrt{\pi} i\eta
    N^{p+3/2}}\left[p_0e^{\frac{2}{3}\pi i}N^{-\frac{1}{3}}\text{Ai}\left(N^{\frac{2}{3}}\zeta e^{\frac{2}{3}\pi i}\right)\right.\right.\label{epsilon p in case 2 error est.}\\
    &\hspace{3.5cm}\left.-q_0e^{\frac{4}{3}\pi
    i}N^{-\frac{2}{3}}\text{Ai}'\left(N^{\frac{2}{3}}\zeta e^{\frac{2}{3}\pi
    i}\right)\right]\Biggr\},
    \end{split}
\end{equation}
where
\begin{eqnarray}
  p_0 &=& \frac{1}{2}\left[g(\sqrt{\zeta},0)+g(-\sqrt{\zeta},0)\right],\label{p0 in case 2 error est.} \\
  q_0 &=&
  \frac{1}{2\sqrt{\zeta}}\left[g(\sqrt{\zeta},0)-g(-\sqrt{\zeta},0)\right].\label{q0 in case 2 error
  est.}
\end{eqnarray}
For a detailed discussion of uniform asymptotic evaluation of
integrals, see \cite[Chap.7]{R.Wong}.\\

Following the same argument, one can derive the asymptotic
approximations of the confluent hypergeometric functions
$\mathbf{M}(aN+1,1,\eta N)$ and $\mathbf{M}'(aN+1,1,\eta N)$.
Indeed, we have
\begin{equation}
\begin{split}
    \mathbf{M}(aN+1,1,\eta N)\sim &  \text{Re }\left\{\frac{e^{-ia\pi N+\frac{1}{2}\eta N}}{\pi i} \left[p^*_0e^{\frac{2}{3}\pi i}N^{-\frac{1}{3}}\text{Ai}\left(N^{\frac{2}{3}}\zeta e^{\frac{2}{3}\pi i}\right)\right.\right.\label{M function in case 2 error est.}\\
    &\hspace{3.2cm}\left.-q^*_0e^{\frac{4}{3}\pi
    i}N^{-\frac{2}{3}}\text{Ai}'\left(N^{\frac{2}{3}}\zeta e^{\frac{2}{3}\pi
    i}\right)\right]\Biggr\},
    \end{split}
\end{equation}
where $\zeta$ is again determined by (\ref{airy transformation in
case 2}) and (\ref{airy transformation mapping of saddle points}),
and is the same as that appearing in (\ref{epsilon p in case 2 error
est.}). The coefficients $p_0^*$ and $q_0^*$ are similar to those
given in (\ref{p0 in case 2 error est.}) and (\ref{q0 in case 2
error est.}), except that the function $h_p(u,\tau)$ in (\ref{g
function in case 2 error estimate}) is replaced by $1$. The
asymptotic formula of $\mathbf{M}'(aN+1,1,\eta N)$ is the same as
(\ref{M function in case 2 error est.}), only with $h_p(u,\tau)$ in
(\ref{g function in case 2 error estimate}) now replaced by $u$.\\

A comparison of (\ref{epsilon p in case 2 error est.}) with (\ref{M
function in case 2 error est.}) and a corresponding formula for
$\mathbf{M}'(aN+1,1,\eta N)$ again establishes (\ref{error estimate when a>0}).\\

\textbf{Case (iii)} In the present case, the contour $\gamma_u$ in
(\ref{epsilonp in bigger than 0}) is shown in Figure \ref{u-sdp2}.
As in case (ii), we divide $\gamma_u$ into two parts, and label the
parts above and below the real axis by $\gamma_u^{(1)}$ and
$\gamma_u^{(2)}$, respectively. Since $\gamma_u^{(1)}$ and
$\gamma_u^{(2)}$ are symmetric with repect to the real axis and
$h_p(\overline{u},\overline{\tau})=\overline{h_p(u,\tau)}$, the
error term $\epsilon_p^+$ can be written as
\begin{equation*}
    \epsilon_p^+=\frac{1}{\pi \eta
    N^{p+1}}\text{Im }\left(\int_{-\infty}^{\infty}\int_{\gamma_u^{(2)}}\frac{h_p(u,\tau)}{u-1}e^{N\psi(u)}e^{-N\tau^2}\mathrm{d}u\mathrm{d}v\right).
\end{equation*}
In view of the equation preceding (\ref{airy transformation in case
2}), we can further write
\begin{equation}
\begin{split}
    \epsilon_p^+=\frac{1}{\pi \eta
    N^{p+1}}\biggr\{&\cos(a\pi
    N)\text{Im }\left(\int_{-\infty}^{\infty}\int_{\gamma_u^{(2)}}\frac{h_p(u,\tau)}{u-1}e^{N\widehat{\psi}(u)}e^{-N\tau^2}\mathrm{d}u\mathrm{d}v\right)\\
    &\left.+\sin(a\pi
    N)\text{Re }\left(\int_{-\infty}^{\infty}\int_{\gamma_u^{(2)}}\frac{h_p(u,\tau)}{u-1}e^{N\widehat{\psi}(u)}e^{-N\tau^2}\mathrm{d}u\mathrm{d}v\right)\right\},\label{case 3 separate into two parts}
\end{split}
\end{equation}
where $\widehat{\psi}(u)=a\ln u-a\ln(1-u)+\eta u$.\\

Now observe that since $N^{\frac{2}{3}}(\eta+4a)$ is unbounded, the
rate in which $\eta\rightarrow -4a$ is controlled; that is, the two
saddle point $u_\pm$ of the phase function $\widehat{\psi}(u)$ do
not coalesce sufficiently fast; cf. (\ref{saddle points of u}). As a
consequence, it is known that the classical method of steepest
descent continues to apply. (Note that $N\widehat{\psi}(u)=N(a\ln
u-a\ln(1-u)+\eta u)$. So, to apply this method, $aN$ must approach
infinity; i.e., $x$ is large.)\\

In (\ref{case 3 separate into two parts}), there are two terms, and
both terms contain the same double integral. Since
$\widehat{\psi}(u)$ is real for $u\in(0,1)$, to get the real part of
the asymptotic approximation of the double integral we need be
concerned with only the horizontal part of the contour
$\gamma_u^{(2)}$. Furthermore, since
$\widehat{\psi}(u_+)>\widehat{\psi}(u_-)$, the real part of the
double integral in (\ref{case 3 separate into two parts}) is
asymptotically equal to
\begin{equation}\label{real part of xxx in case 3}
    \frac{h_p(u_+,0)}{u_+-1}e^{N\widehat{\psi}(u_+)}\sqrt{\frac{2}{-\widehat{\psi}''(u_+)}}\frac{\pi}{N}.
\end{equation}
Since the integrand on the horizontal part of the contour
$\gamma_u^{(2)}$ is real, only the curved part of contour
$\gamma_u^{(2)}$ contributes to the imaginary part of the double
integral. Since the cut for $\widehat{\psi}(u)$ is now from $1$ to
$+\infty$, we can extend the contour $\gamma_u^{(2)}$ into the upper
half-plane by its mirror image, and form the integration path
$\Gamma$ which passes through $u_-$ as an interior point. The
imaginary part of the double integral in (\ref{case 3 separate into
two parts}) is asymptotically equal to
\begin{equation}\label{imaginary part of xxx in case 3}
\frac{1}{2i}\frac{h_p(u_-,0)}{u_--1}e^{N\widehat{\psi}(u_-)}\sqrt{\frac{2}{-\widehat{\psi}''(u_-)}}\frac{\pi}{N}.
\end{equation}
Inserting the contributions given in (\ref{real part of xxx in case
3}) and (\ref{imaginary part of xxx in case 3}) into equation
(\ref{case 3 separate into two parts}) gives
\begin{equation}
\begin{split}
  \epsilon_p^+ =\frac{1}{\eta
  N^{p+2}}\Biggr\{&\sin(a\pi N)\frac{h_p(u_+,0)}{u_+-1}e^{N\widehat{\psi}(u_+)}\sqrt{\frac{2}{-\widehat{\psi}''(u_+)}}\left[1+O(N^{-1})\right]\\
&\left.+\frac{\cos(a\pi
N)}{2i}\frac{h_p(u_-,0)}{u_--1}e^{N\widehat{\psi}(u_-)}\sqrt{\frac{2}{-\widehat{\psi}''(u_-)}}\left[1+O(N^{-1})\right]\right\}.\label{case
3 error est. epsilon}
\end{split}
\end{equation}
Again, since $\widehat{\psi}(u_+)>\widehat{\psi}(u_-)$, (\ref{case 3
error est. epsilon}) implies
\begin{equation}
\begin{split}
  \epsilon_p^+ =\frac{e^{N\widehat{\psi}(u_+)}}{\eta
  N^{p+2}}\Biggr\{&\sin(a\pi N)\frac{h_p(u_+,0)}{u_+-1}\sqrt{\frac{2}{-\widehat{\psi}''(u_+)}}\left[1+O(N^{-1})\right]\\
&+\text{an exponentially small term}\Biggr\}.\label{case 3 error
est. epsilon simplified}
\end{split}
\end{equation}

Asymptotic approximations for $\mathbf{M}(aN+1,1,\eta N)$ and
$\mathbf{M}'(aN+1,1,\eta N)$ can be obtained in a similar manner,
and we have
\begin{equation}
\begin{split}
  \mathbf{M}(aN+1,1,\eta N)=\frac{e^{N\widehat{\psi}(u_+)}}{\sqrt{\pi N}}\Biggr\{&\frac{\sin(a\pi N)}{u_+-1}\sqrt{\frac{2}{-\widehat{\psi}''(u_+)}}\left[1+O(N^{-1})\right]\\
&+\text{an exponentially small term}\Biggr\}\label{case 3 error est.
M function simplified}
\end{split}
\end{equation}
\begin{equation}
\begin{split}
  \mathbf{M}'(aN+1,1,\eta N)=\frac{e^{N\widehat{\psi}(u_+)}}{\sqrt{\pi N}}\Biggr\{&\sin(a\pi N)\frac{u_+}{u_+-1}\sqrt{\frac{2}{-\widehat{\psi}''(u_+)}}\left[1+O(N^{-1})\right]\\
&+\text{an exponentially small term}\Biggr\}.\label{case 3 error
est. M' function simplified}
\end{split}
\end{equation}
The estimate in (\ref{error estimate when a>0}) now follows from
(\ref{case 3 error est. epsilon simplified}), (\ref{case 3 error
est. M function simplified}) and (\ref{case 3 error est. M' function
simplified}).
\begin{rem}
To be more precise, the exponentially small term in (\ref{case 3
error est. M function simplified}) and (\ref{case 3 error est. M'
function simplified}) is\newline
$O\left(e^{N(\widehat{\psi}(u_-)-\widehat{\psi}(u_+)}\right)$.
\end{rem}

\textbf{Case (iv)} Here, we proceed as in case (iii), and reach the
error representation given in (\ref{case 3 separate into two
parts}). There are two terms in this equation; we begin with the one
involving
\begin{equation}\label{case iv real part}
    G_2\equiv \text{Re }\left(\int_{-\infty}^{\infty}\int_{\gamma_u^{(2)}}\frac{h_p(u,\tau)}{u-1}e^{N\widehat{\psi}(u)}e^{-N\tau^2}\mathrm{d}u\mathrm{d}\tau\right).
\end{equation}
As explained in the previous case, for this term we need be
concerned with only the horizontal part of the contour
$\gamma_u^{(2)}$. Since $aN$ is bounded, $a$ must tend to zero and
$\eta$ is negative; see the correspondence between $a$ and $\eta$
given in $\S$\ref{sec 3}. Thus, the saddle point $u_\pm$ behave like
\begin{eqnarray*}
  u_+ &=& \frac{1}{2}-\frac{1}{2}\sqrt{1+\frac{4a}{\eta}}=-\frac{a}{\eta}+O(a^2)\rightarrow 0, \\
  u_- &=& \frac{1}{2}+\frac{1}{2}\sqrt{1+\frac{4a}{\eta}}=1+\frac{a}{\eta}+O(a^2)\rightarrow
  1.
\end{eqnarray*}
As $u\rightarrow 0$, $\widehat{\psi}(u)\sim a\ln u+(a+\eta)u$. This
suggests that we make the transformation
\begin{equation*}
    \widehat{\psi}(u)=a\ln v+(a+\eta)v\equiv \rho(v)
\end{equation*}
with $u=u_+$ corresponding to $v=v_0=-\frac{a}{a+\eta}$. Recall that
$\eta$ is negative and $a\rightarrow 0$; thus, $v_0>0$ and
$v_0\rightarrow 0$. This says that our saddle point $v_0$ coalesces
with an endpoint. We also note that $N\widehat{\psi}(u)=Na\ln
v+N(a+\eta)v$, and that $aN$ is bounded. All these point out that
the steepest descent method does not apply. Therefore, we rewrite
(\ref{case iv real part}) as
\begin{equation}\label{}
    G_2\sim \text{Re
    }\left(\int_{-\infty}^{\infty}\int_{0}^{+\infty}g(v,\tau)v^{aN}e^{N(a+\eta)v}e^{-N\tau^2}\mathrm{d}v\mathrm{d}\tau\right),
\end{equation}
where
\begin{equation}\label{}
    g(v,\tau)=\frac{h_p(u,\tau)}{u-1}\frac{\mathrm{d}u}{\mathrm{d}v}.
\end{equation}
Since
\begin{equation*}
    \left.\frac{\mathrm{d}u}{\mathrm{d}v}\right|_{v=v_0}=\sqrt{\frac{\rho''(v_0)}{\widehat{\psi}''(u_+)}}=\left.\sqrt{\frac{u^2(u-1)^2}{v^2(1-2u)}}\right|_{u=u_+,v=v_0}\rightarrow
    1
\end{equation*}
as $a\rightarrow 0$, we obtain by Watson's lemma
\begin{equation}\label{case iv result for real part}
    G_2=-\frac{h_p(0,0)}{\left[-N(a+\eta)\right]^{aN+1}}\sqrt{\frac{\pi}{N}}\Gamma(aN+1)\left[1+O(N^{-1})\right].
\end{equation}

Next, we consider the first term in (\ref{case 3 separate into two
parts})
\begin{equation}\label{case iv imaginary part}
    G_1\equiv \text{Im }\left(\int_{-\infty}^{\infty}\int_{\gamma_u^{(2)}}\frac{h_p(u,\tau)}{u-1}e^{N\widehat{\psi}(u)}e^{-N\tau^2}\mathrm{d}u\mathrm{d}\tau\right).
\end{equation}
As in case (iii), since the integrand on the horizontal part of the
contour $\gamma_u^{(2)}$ is real, by symmetry we have
\begin{equation}\label{case iv imaginary part deformed}
    G_1=\frac{1}{2i}\left(\int_{-\infty}^{\infty}\int_{L}\frac{h_p(u,\tau)}{u-1}e^{N\widehat{\psi}(u)}e^{-N\tau^2}\mathrm{d}u\mathrm{d}\tau\right),
\end{equation}
where $L$ is the curved part of $\gamma_u^{(2)}$ plus its minor
image in the upper half-plane. Recall that there is a cut for
$\widehat{\psi}(u)$, which extends from $u=1$ to $u=+\infty$ along
the positive real line. Since
$u_-=1+\frac{a}{\eta}+\cdot\cdot\cdot\sim 1$ as $a\rightarrow 1$,
this saddle point coalesces with the branch point at $u=1$. For $u$
near $1$, we write $u=1-v$ with $v$ being small. Then,
\begin{equation*}
    \widehat{\psi}(u)\sim -a\ln v+\eta-(\eta+a)v.
\end{equation*}
This suggests the transformation
\begin{equation}\label{transformation for imaginaray part in case iv}
    \widehat{\psi}(u)= -a\ln z+\eta-(\eta+a)z\equiv
    \widehat{\rho}(z).
\end{equation}
The double integral in (\ref{case iv imaginary part deformed}) now
becomes
\begin{equation}\label{case iv imaginary part after mapping}
    G_1=\frac{1}{2i}\int_{-\infty}^{\infty}\int_{-\infty}^{(0^+)}\frac{h_p(u,\tau)}{u-1}\frac{\mathrm{d}u}{\mathrm{d}z}e^{N\widehat{\psi}(u)}e^{-N\tau^2}\mathrm{d}z\mathrm{d}\tau.
\end{equation}
Since $N\widehat{\rho}(z)=Na\ln z-N(\eta+a)z+N\eta$ and $Na$ is
bounded, the method of steepest descent does not apply. So we
rewrite (\ref{case iv imaginary part after mapping}) in the form
\begin{equation*}
    G_1=\frac{1}{2i}e^{N\eta}\int_{-\infty}^{\infty}\int_{-\infty}^{(0^+)}z^{-aN-1}g(z,\tau)e^{-N(\eta+a)z}e^{-N\tau^2}\mathrm{d}z\mathrm{d}\tau,
\end{equation*}
where
\begin{equation*}
    g(z,\tau)=\frac{z}{u-1}\frac{\mathrm{d}u}{\mathrm{d}z}h_p(u,\tau).
\end{equation*}
Recall that $N(\eta+a)$ is negative. To the inner integral, we can
use Watson's Lemma for loop integrals \cite[p.20]{R.Wong}. Since
\begin{equation*}
    \left.\frac{z}{u-1}\right|_{z=0,u=1}=-1\qquad \text{and}\qquad
    \left.\frac{\mathrm{d}u}{\mathrm{d}z}\right|_{z=0,u=1}=-1,
\end{equation*}
we have $g(0,0)=h_p(1,0)$ and
\begin{equation}\label{result for imaginary part in case iv}
    G_1=\pi h_p(1,0)\sqrt{\frac{\pi}{N}}\frac{\left[-N(a+\eta)\right]^{aN}}{\Gamma(aN+1)}e^{\eta N}\left[1+O(N^{-1})\right].
\end{equation}

Coupling (\ref{case iv result for real part}) and (\ref{result for
imaginary part in case iv}), we have from (\ref{case 3 separate into
two parts})
\begin{equation*}
\begin{split}
  \epsilon_p^+=\frac{1}{\eta N^{p+1}}\biggr\{&\cos(aN\pi)h_p(1,0)\sqrt{\frac{\pi}{N}}\frac{\left[-N(a+\eta)\right]^{aN}}{\Gamma(aN+1)}e^{\eta N}\left[1+O(N^{-1})\right]\nonumber\\
   & \left.-\sin(aN\pi)h_p(0,0)\sqrt{\frac{\pi}{N}}\frac{\Gamma(aN+1)}{\pi[-N(a+\eta)]^{aN+1}}\left[1+O(N^{-1})\right]
   \right\}
   \end{split}
\end{equation*}
which in turn gives
\begin{equation}\label{final in case
   IV}
   \begin{split}
\epsilon_p^+=\frac{1}{\eta
N^{p+1}}\biggr\{&-\frac{\sin(aN\pi)h_p(0,0)}{\pi[-N(a+\eta)]^{aN+1}}\sqrt{\frac{\pi}{N}}\Gamma(aN+1)\left[1+O(N^{-1})\right]\\
&+O(e^{N\eta})\biggr\}.
\end{split}
\end{equation}
Note that $\eta$ is negative.\\

The asymptotics of $\mathbf{M}(aN+1,1,\eta N)$ and
$\mathbf{M}'(aN+1,1,\eta N)$ can again be obtained in a similar
manner. We have
\begin{equation}\label{x fixed, M}
\begin{split}
\mathbf{M}(aN+1,1,\eta
N)=&-\frac{\sin(aN\pi)}{\pi[-N(a+\eta)]^{aN+1}}\Gamma(aN+1)\left[1+O(N^{-1})\right]\\
&+O(e^{N\eta}),
\end{split}
\end{equation}
\begin{equation}\label{x fixed, M'}
\begin{split}
\mathbf{M}'(aN+1,1,\eta
N)=&\frac{a\sin(aN\pi)}{\pi\eta[-N(a+\eta)]^{aN+1}}\Gamma(aN+1)\left[1+O(N^{-1})\right]\\
&+O(e^{N\eta}).
\end{split}
\end{equation}
The estimate (\ref{error estimate when a>0}) follows
immediately from (\ref{final in case IV}), (\ref{x fixed, M}) and (\ref{x fixed, M'}).\\

\textbf{Case (v)} We first consider the subcase in which $a$ is
bounded; that is, $a\in(-M,0)$, where $M>0$ is a constant. Since in
this case only the saddle point $u=a$ matters (see (\ref{psi
function when a<0})), by applying the steepest descent method
directly, we obtain
\begin{equation}\label{}
     \epsilon_p^-=\frac{h_p(a,0)}{N^{p+1/2}}\sqrt{\pi}\left[1+O(N^{-1})\right].
\end{equation}
Since $h_p(u,\tau)$ is analytic along the path of integration, for
bounded $a$ we can easily find a constant $M_p$ satisfying (\ref{error estimate in x<0}).\\

Next, we consider the subcase: $a\rightarrow -\infty$. Motivated by
a method of Olde Daalhuis and Temme \cite{rational}, we introduce
the rational functions
\begin{eqnarray}
  R_0(u,w,a) &=& \frac{1}{u-w} \\
  R_{n+1}(u,w,a) &=&
  \frac{-1}{u-a}\frac{\mathrm{d}}{\mathrm{d}u}\left(uR_n(u,w,a)\right),\quad
  n=0,1,2,\cdot\cdot\cdot,\label{R n rational function}
\end{eqnarray}
where $u,w\in\mathbb{C}$, $u\neq w$ and $u^2\neq a^2$. Using an
induction argument on $n$, it can be readily verified that there are
constants $C_{ij}$, independent of $u$, $w$ and $a$, such that
\begin{equation}\label{expression of R n}
    R_n(u,w,a)=w\sum_{i=0}^{n-1}\sum_{j=0}^{n-1-i}\frac{C_{ij}u^{n-1-j}}{(u-w)^{n+1-i-j}(u-a)^{n+i}},\qquad
    n=1,2,\cdot\cdot\cdot.
\end{equation}

To extend the validity of the asymptotic expansion (\ref{result
a<0}) from bounded $a$ to unbounded $a$, we first define
\begin{equation}\label{}
    \rho_0(a)=\min\{|u-a|:u\text{ is a singularity of }h_0(u)\}.
\end{equation}
Since we shall be interested in the property of $h_p(u,\tau)$ in
$u$, we may suppress the dependence on $\tau$ and simply write
$h_p(u)$ for $h_p(u,\tau)$. To obtain an estimate on $\rho_0(a)$, we
denote by $u_{\pm k}$ the singular points in the $u$-plane, which
are the images of $w_- e^{\pm2k\pi i}$ under the mapping
(\ref{mapping w to u when x<0})-(\ref{relation betwen w and u when
a<0}), where $w_-$ is the saddle point given in (\ref{saddle points
of w}). Thus, we have
\begin{eqnarray}
  \widetilde{f}(t_0(w_-),w_-)&=&a\ln (-a)-a+\gamma, \label{temp 1}\\
  \widetilde{f}(t_0(w_-
e^{\pm2k\pi i}),w_- e^{\pm2k\pi i})&=&a\ln (-u_{\pm k})-u_{\pm
k}+\gamma. \label{temp 2}
\end{eqnarray}
Subtracting (\ref{temp 1}) from (\ref{temp 2}) gives
\begin{equation}\label{equation for rho}
    \pm 2k\pi(a+b-1) i=a\ln(-u_{\pm k})-u_{\pm k}-a\ln(-a)+a.
\end{equation}
For large values of $a$, numerical calculation shows that the
nearest singularities should be $u_{\pm 1}$ which are approximately
equal to $(3.089\pm 7.461 i)a$. Hence $\rho_0(a)\sim -7.75 a$.
(Recall that $a$ is negative). [Put $u_{\pm k}=a v_{\pm k}$ in
(\ref{equation for rho}), and have
\begin{equation*}
    \pm 2k\pi\left(\frac{a+b-1}{a}\right) i=\ln(v_{\pm k})-v_{\pm k}+1.
\end{equation*}
For large $a$, the left-hand side is approximately equal to $\pm
2k\pi i$. Numerical computation of the resulting equation gives $v_{\pm 1}=3.089\pm 7.461 i$.]\\

To estimate the error term $\varepsilon_p^-$ in (\ref{epsilon p}),
we split the loop contour into two parts $\mathcal{L}_1$ and
$\mathcal{L}_2$, and write
\begin{eqnarray}
  \varepsilon_{p,1}^{-} &=& \frac{N^{aN}\Gamma(-a
  N+1)}{2\pi i N^p}\int_{-\infty}^{+\infty}\int_{\mathcal{L}_1}\frac{h_p(u,\tau)}{u}e^{N
\widetilde{\psi}(u)}e^{-N\tau^2}\mathrm{d}u\mathrm{d}\tau, \label{epsilon 1 in a<0}\\
  \varepsilon_{p,2}^{-} &=& \frac{N^{aN}\Gamma(-a
  N+1)}{2\pi i N^p}\int_{-\infty}^{+\infty}\int_{\mathcal{L}_2}\frac{h_p(u,\tau)}{u}e^{N
\widetilde{\psi}(u)}e^{-N\tau^2}\mathrm{d}u\mathrm{d}\tau,
\end{eqnarray}
where $\mathcal{L}_1=\{u\in\text{loop contour}:|u-a|\leq
c|a|^{\theta}\}$, with $0<c<1$ and $\frac{3}{4}\leq \theta\leq 1$. A
simple estimation shows that $\varepsilon_{p,2}^{-}$ is
exponentially
small in comparison with $\varepsilon_{p,1}^{-}$; see \cite[p. 318-320]{rational}. \\

Let $\Gamma_1$ be a closed contour embracing the curve
$\mathcal{L}_1$ such that as $a\rightarrow -\infty$,
\begin{equation}\label{Gamma 1}
    \text{length
    }\Gamma_1=O(|a|^{\theta})\qquad\text{and}\qquad\text{
    distance }(\Gamma_1,\mathcal{L}_1)\sim c_1|a|^{\theta}.
\end{equation}
Let $\Omega_1$ denote the domain bounded by $\Gamma_1$, and its
closure by $\overline{\Omega}_1$. Similarly, let $\Gamma_2$ be a
closed curve embracing $\mathcal{L}_1$ and lying inside $\Gamma_1$
such that as $a\rightarrow -\infty$,
\begin{equation}\label{Gamma 2}
    \text{length }\Gamma_2=O(|a|^{\theta})\qquad\text{
    and}\qquad\text{distance }(\Gamma_2,\mathcal{L}_1)\sim c_2|a|^{\theta}
\end{equation}
with $c_2<c_1$. Let $\Omega_2$ denote the domain bounded by
$\Gamma_2$, and its closure by $\overline{\Omega}_2$. Define
\begin{equation}\label{}
    \widetilde{h}_p=\sup_{w\in\overline{\Omega}_1}|h_p(w)|.
\end{equation}
It is easily verified that as $a\rightarrow -\infty$
\begin{equation}\label{sup bound of h0}
    \sup_{w\in\overline{\Omega}_1}|h_0(w)|\leq
    c_0|h_0(a)|,
\end{equation}
where $c_0$ is another constant independent of $a$. From
(\ref{expression of R n}), we also have
\begin{equation}\label{sup bound of R n}
    \sup_{u\in\Gamma_1,\text{ }w\in\overline{\Omega}_2}|R_p(u,w,a)|\leq
    C_p|a|^{-(2\theta-1)p-\theta},
\end{equation}
where $C_p$ is a constant independent of $a$. By (\ref{sup bound of
h0}), (\ref{sup bound of R n}) and (\ref{Gamma 1}), it follows that
\begin{equation}\label{temp 4}
    \left|\frac{1}{2\pi
    i}\int_{\Gamma_1}R_p(u,w,a)h_0(u)\mathrm{d}u\right|\leq
    C_p'|a|^{-(2\theta-1)p}|h_0(a)|,
\end{equation}
where $C_p'$ is also a constant. By Cauchy's integral formula and
(\ref{hl when a<0}) and (\ref{R n rational function}), we have upon
integration by parts
\begin{equation}
\begin{split}
  h_p(w) &= \frac{1}{2\pi i}\int_{\Gamma_1}R_0(u,w,a)h_p(u)\mathrm{d}u \\
   &= -\frac{1}{2\pi i}\int_{\Gamma_1}g_{p-1}(u)\mathrm{d}\left(uR_0(u,w,a)\right) \\
   &= \frac{1}{2\pi i}\int_{\Gamma_1}(u-a)R_1(u,w,a)g_{p-1}(u)\mathrm{d}u \label{temp6}\\
   &= \frac{1}{2\pi i}\int_{\Gamma_1}R_1(u,w,a)h_{p-1}(u)\mathrm{d}u-\frac{1}{2\pi i}\int_{\Gamma_1}a_{p-1}
   R_1(u,w,a)\mathrm{d}u,
   \end{split}
\end{equation}
where we have used (\ref{6.21 a}) to get
$h_{p-1}(a,\tau)=a_{p-1}(\tau)$. On account of (\ref{sup bound of R
n}), the last integral is
$\widetilde{h}_{p-1}(u)O(|a|^{-(2\theta-1)})$. Repeated application
of (\ref{temp6}) gives
\begin{equation}\label{temp 3}
    h_p(w)=\frac{1}{2\pi
    i}\int_{\Gamma_1}R_p(u,w,a)h_{0}(u)\mathrm{d}u+\widetilde{h}_{p-1}O(|a|^{-(2\theta-1)})+\cdot\cdot\cdot+\widetilde{h}_{0}O(|a|^{-(2\theta-1)p})
\end{equation}
By induction, we have from (\ref{temp 3}) and (\ref{temp 4})
\begin{equation}\label{temp 5}
    \widetilde{h}_p\leq
    \widetilde{C}_p|a|^{-(2\theta-1)p}|h_0(a)|,
\end{equation}
where $\widetilde{C}_p$ is a constant independent of $a$.
Substituting (\ref{temp 5}) in (\ref{epsilon 1 in a<0}), we obtain
\begin{equation}\label{temp 7}
    |\varepsilon_{p,1}^{-}|\leq\frac{A_p|a|^{-(2\theta-1)p}|h_0(a)|\sqrt{\pi}}{N^{p+1/2}},
\end{equation}
where $A_p$ is a constant not depending on $a$. From (\ref{h_0 a<0
analycity}), one can show that $h_0(a)=h_0(a,\tau)\sim \sqrt{-a}$ as
$a\rightarrow-\infty$. Thus, to have the right-hand side of
(\ref{temp 7}) independent of $a$, we must choose $\theta$ so that
$(2\theta-1)\geq\frac{1}{2p}\geq\frac{1}{2}$ for all $p\geq 1$; that
is, $\theta\geq\frac{3}{4}$. With this choice, we can always find a
constant $M_p$, independent of $a$, for each $p$ such that
(\ref{error estimate in x<0}) holds.\\

\section{SUMMARY OF RESULTS}\label{sec 8}

For $a<0$, we have from (\ref{result a<0}) and (\ref{error estimate
in x<0})
\begin{equation}\label{result a<0 final}
      t_n(x,N+1)\sim \frac{(-1)^n\Gamma(n+N+2)N^{-a N}e^{N\gamma}}{\Gamma(n+1)\Gamma(N-n+1)\Gamma(-a
  N+1)}\sum_{l=0}^{\infty}\frac{c_l}{N^{l+\frac{1}{2}}}
\end{equation}
as $N\rightarrow \infty$, where the coefficients $c_l$ are given in
(\ref{cl when a<0}) and the constant $\gamma$ is determined by
(\ref{mapping w to u when x<0}) and (\ref{relation betwen w and u
when a<0}).\\

For $0\leq a\leq\frac{1}{2}$, we have from (\ref{result a>0}) and
(\ref{error estimate when a>0})
\begin{equation}\label{result a>0 final}
\begin{split}
  t_n(x,N+1) \sim \frac{(-1)^n\Gamma(n+N+2)e^{N\gamma}}{\Gamma(n+1)\Gamma(N-n+1)} \Biggr[& \mathbf{M}(a N+1,1,\eta N)\sum_{l=0}^{\infty}\frac{c_l}{N^{l+\frac{1}{2}}}\\
   & \left. +\mathbf{M}'(a N+1,1,\eta N)\sum_{l=0}^{\infty}\frac{d_l}{N^{l+\frac{1}{2}}}\right]
\end{split}
\end{equation}
as $N\rightarrow \infty$, where $\mathbf{M}$ is the confluent
hypergeometric function defined in (\ref{M function in expansion
a>0}), and $\mathbf{M}'$ is its derivative taken with respect to the
third variable. (Note that the second variable in $\mathbf{M}$ is a
constant.) The coefficients $c_l$ and $d_l$ in (\ref{result a>0
final}) are given in (\ref{cl when a>0}) and (\ref{d_l when a>0}),
respectively. The constants $\eta$ and $\gamma$ are determined by
(\ref{mapping w
to u}) and (\ref{the relation between w+- and u+-}).\\

For $a>\frac{1}{2}$, the asymptotic expansion of $t_n(x,N+1)$ can be
obtained by using the symmetry relation (\ref{symmetricity}),
\begin{equation}\label{symmetricity final}
    t_n(aN,N+1)=(-1)^n t_n((1-a)N,N+1).
\end{equation}

The regions of validity of the expansions (\ref{result a<0 final})
and (\ref{result a>0 final}) are, respectively, $a<0$ and $0\leq
a\leq\frac{1}{2}$. However, as we shall show in $\S$\ref{validity
region}, they can be extended to two overlapping regions both
containing a neighborhood of $a=0$.\\

The integral given in (\ref{confluent hypergeometric function for
a>0}) for the confluent hypergeometric function is a typical example
of Airy-type expansion \cite[Chap. VII, $\S$4]{R.Wong}. In fact, we
have already illustrated this point in (\ref{M function in case 2
error est.}), and the region of validity of this formula is
$0<\delta\leq a\leq \frac{1}{2}$. Thus, for $\delta\leq a\leq
\frac{1}{2}$, one can derive from (\ref{result a>0 final}) an
asymptotic approximation for $t_n(x,N+1)$ in terms of the Airy
integral Ai$(\cdot)$ and its
derivative Ai$'(\cdot)$.\\

We now derive from (\ref{result a<0 final}) and (\ref{result a>0
final}) two simple formulas for $t_n(x,N+1)$ when $x$ is fixed,
which in turn means $a=O(1/N)$. First, we consider the case when $x$
is negative, i.e., $a<0$. To calculate the leading coefficient $c_0$
in (\ref{result a<0 final}), we use (\ref{cl when a<0}), (\ref{al
zhankai in a<0}) and (\ref{6.21 a}). Thus, we have
$c_0=a_{0,0}\Gamma\left(\frac{1}{2}\right)$ and
$a_{0,0}=a_0(0)=h_0(a,0)$. From (\ref{h_0 a<0 analycity}), it can be
readily verified that
\begin{equation*}
    h(a,0)=-\frac{\sqrt{2(1-b)}}{b^{3/2}}a+O(a^2)\qquad \text{ as }a\rightarrow
0^-.
\end{equation*}
Hence,
\begin{equation}\label{h(a,0) when a<0 approaches 0}
    c_0=-\frac{\sqrt{2(1-b)\pi}}{b^{3/2}}a+O(a^2)\qquad \quad\text{   as }
    a\rightarrow 0^-.
\end{equation}
To find the value of the constant $\gamma$ in (\ref{result a<0
final}), we recall (\ref{mapping w to u when x<0}) and
(\ref{relation betwen w and u when a<0}), and obtain
\begin{equation*}
    \gamma=\widetilde{f}(t_0(w_-),w_-)-a\ln(-a)+a.
\end{equation*}
Since $t_0(w_-)=t_-$, from (\ref{phase function f for x<0}) and
(\ref{sets of saddle points 2}) it follows that
\begin{equation}\label{gamma when a goes to 0-}
    \gamma=b\ln b+(1-b)\ln(1-b)-2a\ln b+O(a^2)\qquad \text{as
    }a\rightarrow 0^-.
\end{equation}
Inserting (\ref{h(a,0) when a<0 approaches 0}) and (\ref{gamma when
a goes to 0-}) in (\ref{result a<0 final}) gives
\begin{equation}\label{fixed x, expansion, a<0}
    t_n(x,N+1)=\frac{(-1)^{n+1}\Gamma(n+N+2)N^x n^{-2x-2}}{\Gamma(N+1)\Gamma(-x)}\left[1+O\left(\frac{1}{N}\right)\right]
\end{equation}
for fixed $x<0$.\\

Next, we consider the case when $x$ is nonnegative; this means that
$0\leq a\leq \frac{1}{2}$ and $aN$ is bounded. Substituting (\ref{x
fixed, M}) and (\ref{x fixed, M'}) in (\ref{result a>0 final}) gives
\begin{equation}\label{temp in fixed x when a>0}
\begin{split}
  t_n(x,N+1) =& \frac{(-1)^n\Gamma(n+N+2)e^{N\gamma}}{\Gamma(n+1)\Gamma(N-n+1)}\left\{  \frac{\Gamma(x+1)\sin\pi
    x}{\sqrt{N}\left[-N(a+\eta)\right]^{aN+1}\pi}\right. \\
    &
    \left.\times\left[-c_0+\frac{a}{\eta}d_0\right]\left[1+O\left(\frac{1}{N}\right)\right]+O(e^{N\eta})\right\}.
    \end{split}
\end{equation}
The coefficients $c_0$ and $d_0$ can be calculated by using (\ref{cl
when a>0}), (\ref{d_l when a>0}), (\ref{saddle points of u}) and
(\ref{a0 and b0}); the result is
\begin{eqnarray*}
  -c_0+\frac{a}{\eta}d_0 &=& -\sqrt{\pi}\left[a_{0,0}-\frac{a}{\eta}b_{0,0}\right]\sim -\sqrt{\pi}\left[a_{0,0}+u_+b_{0,0}\right]\nonumber \\
   &=& -\sqrt{\pi}h_0(u_+,0).
\end{eqnarray*}
From (\ref{h function when a>0 in analyticity section}), it can be
shown that
\begin{equation}\label{temp fixed x a}
    h_0(u_+,0)=-\frac{2\eta\sqrt{1-b}}{b^{3/2}}+O(a)\qquad
    \text{as }a\rightarrow 0^+.
\end{equation}
Hence,
\begin{equation}\label{coefficient for fix x a>0 in last section}
    -c_0+\frac{a}{\eta}d_0=\frac{2\sqrt{\pi}\eta\sqrt{1-b}}{b^{3/2}}+O(a)\qquad
    \text{as }a\rightarrow 0^+.
\end{equation}
By (\ref{sets of saddle points 1}) and (\ref{sets of saddle points
2}), we also have
\begin{equation}\label{saddle points as a goes to 0}
\begin{split}
 w_+&=1-\frac{a}{b^2}+O(a^2),\qquad \qquad \qquad \hspace{0.6cm} w_{-}=\frac{a}{b^2}+O(a^2),\\
    t_{+}&= \frac{1}{1+b}-\frac{1}{1+b}a+O(a^2),
        \qquad \qquad t_{-}=(1-b)+(1-b)a+O(a^2).
        \end{split}
\end{equation}
Inserting (\ref{saddle points as a goes to 0}) into (\ref{t+- and
w+- and mapping in proof II}), and subtracting the two resulting
equations, we obtain
\begin{equation}\label{eta when a goes to 0}
\begin{split}
  \eta = & -[(1-b)\ln(1-b)+(1+b)\ln(1+b)] \\
   & -2a \ln\frac{(1-b)\ln(1-b)+(1+b)\ln(1+b)}{b^2}+O(a^2).
   \end{split}
\end{equation}
Solving $\gamma$ in (\ref{t+- and w+- and mapping in proof II}) also
gives
\begin{equation}\label{gamma when a goes to 0+}
\begin{split}
  \gamma = & b\ln b+(1-b)\ln(1-b) \\
   &+a \ln\frac{(1-b)\ln(1-b)+(1+b)\ln(1+b)}{b^2}+O(a^2).
   \end{split}
\end{equation}
A combination of (\ref{coefficient for fix x a>0 in last section}),
(\ref{eta when a goes to 0}), (\ref{gamma when a goes to 0+}) and
(\ref{temp in fixed x when a>0}) yields
\begin{equation}\label{result for fixed x>0}
\begin{split}
  t_n(x,N+1)=&\frac{(-1)^{n+1}\Gamma(n+N+2)N^x n^{-2x-2}\Gamma(x+1)}{\Gamma(N+1)\pi}
    \\
   & \times\left\{\sin\pi x
   \left[1+O\left(\frac{1}{N}\right)\right]+O(e^{N\eta})\right\};
   \end{split}
\end{equation}
cf. (\ref{fix temp 1}).

\section{REGIONS OF VALIDITY FOR EXPANSIONS (\ref{result a<0 final}) AND (\ref{result a>0 final})}\label{validity region}

We first consider the case of expansion (\ref{result a<0 final}),
which is valid for $a<0$. Similarity of the two asymptotic formulas
in (\ref{fixed x, expansion, a<0}) and (\ref{result for fixed x>0})
suggests that the region of validity of expansion (\ref{result a<0
final}) may be extended beyond the origin. This is indeed the case,
and we shall show that expansion (\ref{result a<0 final}) actually
holds for $a\in(-\infty, a_--\delta]$ for any small number
$\delta>0$, $a_-$ being the point on the positive real line given in
(\ref{critical value of a}), where the two saddle points $w_\pm$
coalesce with each other.\\

To prove this, we recall that when $0\leq a\leq a_--\delta$, we have
$f(t_0(w_+),w_+)<f(t_0(w_-),w_-)$; see a statement following
(\ref{SDP equation for w+}). In other words, in the case $0\leq
a\leq a_--\delta$, the contribution from the saddle point $w_+$ is
exponentially small in comparison with that from the saddle point
$w_-$. So, up to an exponentially small term (E.S.T), we need deal
with only one saddle point, namely, $w_-$. By an argument similar to
that in case (iii) of $\S$\ref{sec 7} (see, especially, equation
(\ref{case 3 separate into two parts})), the integral representation
for $t_n(x,N+1)$ given in (\ref{IP after the t mapping}) can now be
written as
\begin{equation}\label{ip after the t mapping, a>0 Gamma}
\begin{split}
t_n(x,N+1)=& \frac{(-1)^n}{\pi}\frac{\Gamma(n+N+2)}{\Gamma(n+1)\Gamma(N-n+1)}\\
&\times\int_{-\infty}^{+\infty}\left[\sin (aN\pi)
\int_0^{w_+}\frac{1}{w-1}e^{N
\widehat{f}({t_0(w)},w)}\mathrm{d}w+\text{E.S.T}\right]\frac{\mathrm{d}t}{\mathrm{d}
\tau}e^{-N\tau^2}\mathrm{d}\tau,
\end{split}
\end{equation}
where
\begin{equation}
\begin{split}
    \widehat{f}({t_0(w)},w)=& b \ln(1-t_0(w))+(1-b)\ln t_0(w)\\
    &+a\ln w-a \ln(1-w)+b \ln[1-(1-t_0(w))w],
    \end{split}
\end{equation}
which is real on the integration path of $w$. [Note that the term
$a\ln(w-1)$ in $f(t_0(w),w)$ is replaced by $a\ln(1-w)$ in
$\widehat{f}({t_0(w)},w)$.] Similar to (\ref{mapping w to u when
x<0}), we make the transformation $w\rightarrow u$ defined by
\begin{equation}\label{mapping w to u when x>0 gamma}
    \widehat{f}(t_0(w),w)=a\ln (u)-u+\gamma,
\end{equation}
with
\begin{equation}\label{relation betwen w and u when a>0 gamma}
    u(w_-)=a.
\end{equation}
Inserting (\ref{mapping w to u when x>0 gamma}) into (\ref{ip after
the t mapping, a>0 Gamma}) gives
\begin{equation}\label{IR after mapping in x>0 Gamma}
\begin{split}
    t_n(x,N+1)=&\frac{(-1)^n}{\pi}\frac{\Gamma(n+N+2)}{\Gamma(n+1)\Gamma(N-n+1)}\int_{-\infty}^{+\infty}\biggr[\sin (aN\pi) e^{N\gamma}\\
& \times\left. \int_{0}^{\infty}\frac{h(u,\tau)}{u} e^{N \left(a\ln
(u)-u\right)}\mathrm{d}u+\text{E.S.T}\right]e^{-N\tau^2}\mathrm{d}\tau,
\end{split}
\end{equation}
where
\begin{equation}\label{h function when a>0 gamma}
    h(u,\tau)=\frac{u}{w-1}\frac{\mathrm{d} w}{\mathrm{d}
u}\frac{\mathrm{d} t}{\mathrm{d} \tau}.
\end{equation}
The derivative $\mathrm{d} w/\mathrm{d} u$ is again given by
(\ref{dw over du when a<0}) for $u\neq a$ and (\ref{dw over du when
a<0 saddle point}) for $u=a$. Let $h_0(u,\tau)\equiv h(u,\tau)$, and
define the sequence $\left\{h_l(u,\tau)\right\}$ as in (\ref{hl when
a<0}). Following the procedure as in case (ii) of $\S$\ref{sec 6},
we arrive at
\begin{equation}\label{result a>0 gamma}
\begin{split}
      t_n(x,N+1)\sim &\frac{(-1)^n\Gamma(n+N+2)}{\Gamma(n+1)\Gamma(N-n+1)}\\
      &\times \left[\frac{N^{-a N}e^{N\gamma}\Gamma(a
  N)\sin(aN\pi)}{\pi}\sum_{l=0}^{\infty}\frac{c_l}{N^{l+\frac{1}{2}}}+\text{E.S.T}\right],
\end{split}
\end{equation}
where the coefficients $c_l$ are as given in (\ref{cl when a<0}).
Although the constant $\gamma$ in (\ref{result a>0 gamma}) is
determined by the mapping defined by (\ref{mapping w to u when x>0
gamma}) and (\ref{relation betwen w and u when a>0 gamma}) while the
corresponding constant $\gamma$ in (\ref{result a<0 final}) is
determined by a different mapping defined by (\ref{mapping w to u
when x<0}) and (\ref{relation betwen w and u when a<0}), the two
constants are actually the same and can be defined by a single
equation
\begin{equation}
    \gamma=\text{Re }(\widetilde{f}(t_0(w_-),w_-)-a\ln (-a))-a,
\end{equation}
where $\widetilde{f}(t_0(w),w)$ is the same function given in
(\ref{mapping w to u when x<0}) or (\ref{phase function f for x<0}).
[Recall that $\gamma$ is a real number by Theorem \ref{thm 1} in
$\S$\ref{sec 4}.] Thus, the two expansions (\ref{result a<0 final})
and (\ref{result a>0 gamma}) are exactly the same, except for an
exponentially small term. The region of validity for expansion
(\ref{result a<0 final}) is therefore extended from $(-\infty, 0)$
to $(-\infty, a_--\delta]$, $0<\delta<a_-$.\\

Next, we consider the case of expansion (\ref{result a>0 final}). In
the derivation of expansion (\ref{result a<0 final}) when $a<0$, we
have used a mapping $w\rightarrow u$ defined by (\ref{mapping w to u
when x<0}). The right-hand side of this equation is the phase
function of the Hankel loop integral for Gamma function; see
(\ref{gamma function in use}). But, as an alternative, we can also
use the phase function in the integral representation of the
confluent hypergeometric function, as in what we have done when
$0\leq a\leq \frac{1}{2}$. This will enable us to extend the
validity of expansion (\ref{result a>0 final}) to negative values of
$a$. To prove this, we mention that for Re$c>$Re$d$, instead of
(\ref{m function a>0}), we have
\begin{equation}\label{m function a<0}
    M(d,c,z)=-\frac{\Gamma(c)\Gamma(1-d)}{2\pi
    i\Gamma(c-d)}\int_{\gamma_2}(-u)^{d-1}(1-u)^{c-d-1}e^{zu}\mathrm{d}u,
\end{equation}
where $\gamma_2$ is the integration path given in (\ref{integral
temp 4}). This integral representation can be established in the
same manner as (\ref{integral temp 4}) or (\ref{integral temp 3}).
Thus, with $d=aN+1$ ($a<0$), $c=1$ and $z=\eta N$,
$\textbf{M}(d,c,z):=M(d,c,z)/\Gamma(c)$ can be written as
\begin{equation}\label{confluent hypergeometric function for a<0}
    \textbf{M}(aN+1,1,\eta N)=\frac{1}{2\pi
    i}\int_{\gamma_2}e^{N\widehat{\psi}(u)}\frac{\mathrm{d}u}{u-1},
\end{equation}
where $\widehat{\psi}(u)=a\ln(-u)-a\ln(1-u)+\eta u$. The saddle
points $u_\pm$ of $\widehat{\psi}(u)$ are the same as those of
$\psi(u)$ given in (\ref{saddle points of u})
\begin{equation}\label{saddle points of u a<0}
    u_\pm=\frac{\eta\pm\sqrt{\eta^2+4a\eta}}{2\eta}.
\end{equation}
It can be shown that the movements of $u_\pm$ when $a<0$ and
$\eta<0$ are like those of the saddle points $w_\mp$ of
$\widetilde{f}(t_0(w),w)$ when $a<0$. Also, we have
\begin{equation}\label{M' function in expansion a<0}
 \textbf{M}'(aN+1,1,\eta N)=\frac{1}{2\pi i}
 \int_{\gamma_2} e^{N\widehat{\psi}(u)} \frac{u}{u-1}\mathrm{d}u,
\end{equation}
where the derivative is taken with respect to the third variable in
$\textbf{M}(d,c,z)$.\\

Now, define the mapping $w\rightarrow u$ by
\begin{equation}
    \widetilde{f}(t_0(w),w)=\widehat{\psi}(u)+\gamma=a \ln(-u)-a\ln(1-u)+\eta u+\gamma\label{mapping w to u confluent
a<0}
\end{equation}
with
\begin{equation}\label{the relation between w+- and u+- confluent a<0}
    u(w_+^{(\pm)})=u_-^{(\pm)},\quad\quad\quad\quad u(w_-)=u_+,
\end{equation}
where $\eta$ and $\gamma$ are real numbers. (Note that, similar to
(\ref{correspondence on upper edge}) and (\ref{correspondence on
lower edge}), since $w_+$ is real and $w_+>1$, we have $w_+$ on both
upper and lower edges of the branch cut $[0,+\infty)$ along the real
line. We denote these points by $w_+^{(\pm)}$, respectively. The
situation is the same with $u_-^{(\pm)}$). Coupling (\ref{mapping w
to u confluent a<0}) and (\ref{IP after the t mapping a<0}), the
double integral representation for $t_n(x,N+1)$ takes the canonical
form
\begin{equation}\label{IP for x<0 deformed confluent}
\begin{split}
    t_n(x,N+1)=&\frac{(-1)^n}{2\pi
i}\frac{\Gamma(n+N+2)e^{N\gamma}}{\Gamma(n+1)\Gamma(N-n+1)}\\
&\times
\int_{-\infty}^{+\infty}\int_{\gamma_u}\frac{h(u,\tau)}{u-1}e^{N
\widehat{\psi}(u)-N\tau^2}\mathrm{d}u\mathrm{d}\tau,
\end{split}
\end{equation}
instead of the one given in (\ref{IR after mapping in x<0}). In
(\ref{IP for x<0 deformed confluent}),
\begin{equation}\label{h function a<0 confluent}
h(u,\tau)=\frac{u-1}{w-1}\frac{\mathrm{d} w}{\mathrm{d}
u}\frac{\mathrm{d} t}{\mathrm{d} \tau},
\end{equation}
and $\gamma_u$ is the steepest descent path of $\widehat{\psi}(u)$
passing through $u_+$. Let $h_0(u,\tau)\equiv h(u,\tau)$, and define
the sequence $\left\{h_n(u,\tau)\right\}$ as in
(\ref{h_l})-(\ref{expansion of coeffiencet an and bn}). Following
the same argument as in Case 1 ($0\leq a\leq \frac{1}{2}$) of
$\S$\ref{sec 6}, we obtain
\begin{equation}\label{result a<0 confluent}
\begin{split}
  t_n(x,N+1) \sim  \frac{(-1)^n\Gamma(n+N+2)e^{N\gamma}}{\Gamma(n+1)\Gamma(N-n+1)}&\left[\mathbf{M}(a N+1,1,\eta N)\sum_{l=0}^{\infty}\frac{c_l}{N^{l+\frac{1}{2}}}\right.\\
   & \left. +\mathbf{M}'(a N+1,1,\eta
   N)\sum_{l=0}^{\infty}\frac{d_l}{N^{l+\frac{1}{2}}}\right],
\end{split}
\end{equation}
where $c_l$ and $d_l$ are as given in (\ref{c_l}) and (\ref{d_l}), respectively.\\

Although the constants $\eta$ and $\gamma$ in (\ref{result a<0
confluent}) and (\ref{result a>0 final}) are determined by two
different sets of equations, namely, (\ref{mapping w to u confluent
a<0})-(\ref{the relation between w+- and u+- confluent a<0}) and
(\ref{mapping w to u})-(\ref{the relation between w+- and u+-}),
respectively, one can use a single system of equations to determine
these two constants for $a\in(-\infty, \infty)$. Indeed, we can
define them by
\begin{equation}\label{temp temp}
\left\{\begin{split}
    \text{Re }f(t_0(w_-),w_-)&=\text{Re }\psi(u_+)+\gamma,\\
    \text{Re }f(t_0(w_+),w_+)&=\text{Re }\psi(u_-)+\gamma,
\end{split}\right.
\end{equation}
where $f(t_0(w),w)$ and $\psi(u)$ are the same as those given in
(\ref{mapping w to u}). (If $w_-$ and $u_+$, or $w_+$ and $u_-$, are
on the branch cut, we should choose the correspoinding saddle points
either both on the upper edge or both on the lower edge of the real
line in the equations of (\ref{temp temp})). Thus, (\ref{result a<0
confluent}) and (\ref{result a>0 final}) are exactly the same,
including the coefficients in these expansions, and the region of
validity of expansion (\ref{result a>0 final}) has been extended
from $0\leq
a\leq \frac{1}{2}$ to $-\infty<a\leq \frac{1}{2}$.\\

\section{SMALL AND LARGE ZEROS}\label{sec 9}
It is well-known that from an asymptotic approximation of a
function, one can deduce an asymptotic approximation for its zeros.
The easiest way to achieve this is to apply the following result
given in Hethcote \cite{Herbert}:
\begin{lem}\label{lemma zeros}
In the interval $[a-\rho, a+\rho]$, suppose
$f(t)=g(t)+\varepsilon(t)$, where $f(t)$ is continuous, $g(t)$ is
differentiable, $g(a)=0$, $m=\min|g'(t)|>0$, and
\begin{equation*}
    E=\max|\varepsilon(t)|<min\left\{|g(a-\rho)|,|g(a+\rho)|\right\}.
\end{equation*}
Then, there exists a zero $c$ of $f(t)$ in the interval such that
\begin{equation*}\label{==}
    |c-a|\leq\frac{E}{m}.
\end{equation*}
\end{lem}

To state our result, we first recall the numbers $a_\pm$ given in
(\ref{critical value of a})
\begin{equation*}
    a_\pm=\frac{1\pm\sqrt{1-b^2}}{2},
\end{equation*}
where $b$ is a fixed number in $(0,1)$. Clearly, $0\leq
a_-<\frac{1}{2}<a_+<1$. The number $b$ can be as close to $0$ or $1$
as it may be, but it is fixed. Hence, $a_-$ can be close to $0$ but
does not tend to $0$. Similarly, $a_+$ can be close to $1$ but
does not approach $1$.\\

\begin{thm} Let $x_{N+1,s}$ denote the $s$th zero of the discrete Chebyshev
polynomial $t_n(x,N+1)$, arranged in ascending order
$0<x_{N+1,1}<x_{N+1,2}<\cdot\cdot\cdot<x_{N+1,n-1}<x_{N+1,n}$.\\ If
$s-1<a_-N$, then we have
\begin{equation}\label{temp in sec 10}
    x_{N+1,s}=(s-1)+O\left\{e^{N[\widehat{\psi}(u_-)-\widehat{\psi}(u_+)]}\right\}
\end{equation}
as $N\rightarrow\infty$, where
\begin{equation}\label{}
    \widehat{\psi}(u_-)-\widehat{\psi}(u_+)=2a\ln\frac{1+\sqrt{1+4a/\eta}}{1-\sqrt{1+4a/\eta}}+\eta\sqrt{1+4a/\eta}.
\end{equation}
If $s>n-(1-a_+)N$, then we have
\begin{equation}\label{}
\begin{split}
    x_{N+1,s}&=N-x_{N+1,n-s+1}\\
    &=N-n+s+O\left\{e^{N[\widehat{\psi}(u_-)-\widehat{\psi}(u_+)]}\right\}
    \end{split}
\end{equation}
as $N\rightarrow \infty$.
\end{thm}
\begin{proof}
The proof is quite straightforward. For $a<0$, we know from
(\ref{result a<0}) that the discrete Chebyshev polynomials do not
have zeros. For $0\leq a<a_-$, it follows from (\ref{case 3 error
est. M function simplified}), (\ref{case 3 error est. M' function
simplified}) that both $\textbf{M}$ and $\textbf{M}'$ have zeros at
approximately $x=m$ with a exponentially small error, where $m$ is a
nonnegative integer; see the remark following (\ref{case 3 error
est. M function simplified}). The results now are obtained by a
direct application of Lemma \ref{lemma zeros} and the symmetry
relation (\ref{symmetricity}). Also, note that for $0\leq a<a_-$,
the exponent $N[\widehat{\psi}(u_-)-\widehat{\psi}(u_+)]$ in
(\ref{temp in sec 10}) is negative, unless
$\sqrt{1+4a/\eta}\rightarrow 0$ in which case $a\rightarrow a_-$
(i.e., $x$ is no longer small).
\end{proof}

\bibliographystyle{plain}

\end{document}